\documentclass[a4paper]{amsart}

\usepackage{verbatim,amsmath,amssymb,skak,amscd}
\usepackage[displaymath,pagewise]{lineno} %for line numbers switch*,

\newtheorem{thm}{Theorem}
\newtheorem*{thmA}{Theorem A}
\newtheorem*{thmB}{Theorem B}

\newtheorem{prop}[thm]{Proposition}
\newtheorem{lem}[thm]{Lemma}
\newtheorem{cor}[thm]{Corollary}

\numberwithin{equation}{section} \numberwithin{thm}{section}

\newtheorem{definition}[thm]{Definition}
\newtheorem{conjecture}[thm]{Conjecture}

\newcommand*\patchAmsMathEnvironmentForLineno[1]{%
  \expandafter\let\csname old#1\expandafter\endcsname\csname #1\endcsname
  \expandafter\let\csname oldend#1\expandafter\endcsname\csname end#1\endcsname
  \renewenvironment{#1}%
     {\linenomath\csname old#1\endcsname}%
     {\csname oldend#1\endcsname\endlinenomath}}%
\newcommand*\patchBothAmsMathEnvironmentsForLineno[1]{%
  \patchAmsMathEnvironmentForLineno{#1}%
  \patchAmsMathEnvironmentForLineno{#1*}}%
\AtBeginDocument{%
\patchBothAmsMathEnvironmentsForLineno{equation}%
\patchBothAmsMathEnvironmentsForLineno{align}%
\patchBothAmsMathEnvironmentsForLineno{flalign}%
\patchBothAmsMathEnvironmentsForLineno{alignat}%
\patchBothAmsMathEnvironmentsForLineno{gather}%
\patchBothAmsMathEnvironmentsForLineno{multline}%
}

\newcommand{\cal}{\mathcal}

\newcommand{\rd}{{\mathbb R}^d}

\newcommand{\R}{{\mathbb R}}

\newcommand{\nor}{{\rm nor}\,}
\newcommand{\supp}{{\rm spt}\,}
\newcommand{\spt}{{\rm spt}\,}
\newcommand{\diam}{{\rm diam}\,}

\newcommand{\E}{{\mathbb E}}
\newcommand{\EE}{{\mathbf E}}
\newcommand{\FFF}{{\mathbf F}}
\newcommand{\CC}{{\mathbf C}}

\newcommand{\clos}{\operatorname{clos}}
\newcommand\inv{^{-1}}
\newcommand\setd{\partial}

\newcommand\mass{\operatorname{mass}}

\newcommand{\eps}{\varepsilon}

\newcommand{\restrict}[2]{\left. #1 \right|_{#2}}

\newcommand{\Rd}{\mathbb{R}^d}
\newcommand{\Rn}{\mathbb{R}^n}
\newcommand{\barsod}{\overline{SO(d)}}
\newcommand{\barson}{\overline{SO(n)}}

\def\en{\mathbb N}
\def\er{\mathbb R}

\def\C{\mathcal C}

\def\F{\mathcal F}

\def\D{\mathbb D}
\def\I{\mathbb I}
\def\E{\mathcal E}
\def\F{\mathcal F}
\def\J{\mathcal J}

\newcommand{\graph}{\operatorname{graph}}

\newcommand{\proj} {\nu}
\newcommand{\csubset}{\subset \subset}
\newcommand{\noreps}{{\rm nor}_\eps}

\newcommand{\vol}{\operatorname{Vol}}
\newcommand{\wth}{\operatorname{width}}
\newcommand{\z}{\operatorname{(zero-section)}}
\newcommand{\epi}{\operatorname{epi}}

\newcommand\DC{\text{DC}}
\newcommand\WDC{\text{WDC}}
\newcommand\MA{\text{MA}}

\newcommand\bdry{\operatorname{bdry}}

\begin{document}

\title[Kinematic formulas for WDC sets]{Kinematic formulas for sets defined by differences of convex functions}
%[DC integral geometry]{DC integral geometry}
%[Kinematic formulas for sets defined by d.c. functions]{Kinematic formulas for sets defined by differences of convex functions}
\author{Joseph H.G. Fu, Du\v san Pokorn\'y, and Jan Rataj}
\thanks{JHGF was supported by NSF grant DMS-1406252. The second author is a junior researcher in the University Centre for Mathematical Modelling, Applied Analysis and Computational Mathematics (MathMAC). The second and third authors were also supported by grant GA\v CR~15-08218S}

\begin{abstract} The class $ \WDC(M)$  consists of all subsets of a smooth manifold $M$ that may be expressed in local coordinates as certain sublevel sets of DC (differences of convex) functions. If $M$ is Riemanian and $G$ is a group of isometries acting transitively on the sphere bundle $SM$, we define the invariant curvature measures of  compact \WDC~ subsets of $M$, and show that pairs of such subsets are subject to the array of kinematic formulas known to apply to smoother sets. Restricting to the case $(M, G) = (\Rn, \barson)$, this extends and subsumes Federer's theory of sets with positive reach in an essential way. The key technical point is equivalent to a sharpening of a classical theorem of Ewald, Larman, and Rogers characterizing the dimension of the set of directions of line segments lying in the boundary of a given convex body.

\end{abstract}

\email{joefu@uga.edu}
\email{dpokorny@karlin.mff.cuni.cz}
\email{rataj@karlin.mff.cuni.cz}

\date{\today}
\maketitle

\section{Introduction} 
The classical  {\it principal kinematic formula} (PKF) expresses, in terms of geometric quantities (intrinsic volumes) associated separately to compact subsets $A,B \subset \Rd$, the integral of the Euler characteristic of the intersection $A\cap \gamma B$ over $\gamma\in SO(d)$. However, it is necessary to restrict $A,B$ to have ``reasonable" smoothness: the original framework of Blaschke assumed $A,B$ to be convex, and subsequently Santal\'o and Chern \cite{santalo, chern} showed that the formula holds when $A,B$ are smooth domains. Both these cases were subsumed by  the theory of Federer \cite{cm}, treating the case of sets $A,B$ of {\it positive reach}. This theory has represented the state of the art for many years: the extensions by Hadwiger to sets from the ``convex ring" (the class of finite unions of compact convex bodies), and in \cite{zahle} to the class $U_{PR}$ of finite unions of sets with positive reach in general position, both rely on the analysis of the convex/positive reach case; and
the extension of \cite{fu94} to subanalytic sets relies on the very special finiteness properties that these sets enjoy (in fact the methods there apply also to sets definable with respect to any given o-minimal  structure \cite{vdD}).

It is natural to ask to characterize precisely the minimal amount  of smoothness needed to ensure that the PKF holds. This question turns out to be  subtle and elusive, and indeed it appears to evade all classical smoothness classes. The paper \cite {fu94} attempted to formulate  an answer using a notion of smoothness arising from the apparatus of the  proof of the PKF itself. The basic object of interest is the {\it normal cycle} $N(A)$ of $A \subset \Rd$, viz. an integral current associated to a singular subspace $A$ that stands in for the manifold of unit normals of a smooth set $A$. Closely related is the {\it differential cycle} $\D(f)$ of a  nonsmooth function $f:\Rd \to \R$, which is an integral current that stands in for the graph of the differential of a smooth $f$. A function $f$ that admits such a differential cycle is  called a {\it Monge-Amp\`ere (MA) function}. The general theory of \cite{fu94} posits that a set $A$ subject to the PKF should be given as a sublevel set of an MA function at a {\it weakly regular} value (cf. Definition \ref{def:weakly reg} below). Unfortunately, the theory found only limited success: in order to prove the PKF for pairs of such sets  it was necessary to introduce additional ad hoc hypotheses on the supports of $N(A)$ and $\D(f)$ (viz. the hypotheses on $\nor(f,0),\nor(g,0)$ in Theorem 2.2.1 of \cite{fu94}).

The main point of the present paper is to show that the general scheme of \cite{fu94} works completely, without these ad hoc devices, in the case of the \WDC~ sets introduced in \cite{PR}. In this last work it was shown first of all that any \DC~  function $f$ (i.e. a function expressible locally as $f = g-h$, where $g,h$ are convex) is \MA. A set $A$ is \WDC~ if it may be expressed as a sublevel set of a \DC~ function $f$ at a weakly regular value. In the present paper, by sharpening  a theorem of Ewald, Larman, and Rogers \cite{ELR}, and using  a construction of Pavlica and Zaj\'i\v cek \cite{PZ},  we show that the unwanted ad hoc hypotheses are always fulfilled in this setting.

Using a characterization of sets with positive reach due to Kleinjohann and Bangert (Theorem \ref{thm:klein-bang} below), it is easy to see that any set with positive reach is a \WDC~ set. Since \WDC~ is closed under finite unions and intersections in general position, it follows that any $U_{PR}$ set (in the original sense of \cite{zahle}) is also \WDC.  Thus the theory developed here subsumes that of \cite{cm, zahle}, and indeed ventures well beyond it, covering for example in a systematic way also the case   of general convex hypersurfaces.

%
%
%In the present paper we show that these technical conditions are always fulfilled for pairs  of \WDC~ sets $A,B$. Here, following \cite{PR}, a set $A\subset \Rd$ is ~\WDC~ if $A $ is a  weakly regular sublevel sets of a ~\DC functions $f$, i.e. a function that may be expressed in local coordinates as $f=g-h$, where $g,h$ are convex.
%Therefore  kinematic formulas hold for such pairs. The first point, already established in \cite{PR}, is that any \DC~ function $f$ belongs to the class \MA. The next essential point is to show that $\nor(A) \times \nor(B)$ has the right dimension. This was established in a weak form in \cite{PR}; the strong form (an immediate consequence of Theorem \ref{WDCbundle} below) is the main technical point of the present paper. With the exception of the subanalytic sets, the class ~\WDC~ includes all of the integral geometric regularity classes above, as well as other natural objects such as general convex hypersurfaces (boundaries of convex bodies).

\subsection {Plan of the paper}  In fact there are many kinematic formulas beyond the PKF, which may be treated together in terms of integration of invariant forms over the normal cycle. It is known (\cite{fu90}, \cite{ale-be}) that if $M$ is Riemannian and admits a Lie group $G$ of isometries that acts transitively on the tangent sphere bundle $SM$ (i.e. $(M,G)$ is a {\bf Riemannian isotropic space}) then kinematic formulas exist for pairs of subsets of $M$ belonging to any of the classical integral geometric regularity classes above \cite{fu90}. We will carry out our analysis of ~\WDC~ sets in this same general context.

Section \ref{sect:generalities} is devoted to  notions that will figure importantly in the subsequent discussion. In particular we recall from \cite{PR} the definitions of DC and MA functions and the key inclusion $\DC \subset \MA$. We give a new result (Theorem \ref{thm:ma addition}) describing the differential cycle of the sum of two \MA~ functions in general position on a homogeneous space; this formula is key to the subsequent proof of the kinematic formulas.
 We recall also the definition of a WDC~ set $A$ as the level set $A= f\inv(0)$ of a nonnegative DC~ function $f$ for which $0$ is a weakly regular value. In this case we say that $f$ is a DC {\bf aura} for $A$, and we recall the main result of \cite{PR}, giving the construction of the conormal cycle of $A$ in terms of $f$ and showing that it is independent of the choice of aura $f$.

The main result (Theorem \ref{WDCbundle}) of Section 3  states that in this setting the sets $\noreps(f)$, defined as the set of all elements of  the sphere bundle $SM$) that arise as normalized Clarke differentials of $f$ at $f=0$, has locally finite $(d-1)$-dimensional Minkowski content, where $d = \dim M$. Since Minkowski content--- unlike Hausdorff measure--- behaves well under Cartesian products, this fact is the key to establishing the support condition needed to prove the kinematic formulas.
The proof of Theorem \ref{WDCbundle} relies ultimately on a fundamental construction of Pavlica and Zaj\'i\v cek \cite{PZ} relating the support elements of the graph of a DC function $f:\Rd\to \R$ to the set of  support hyperplanes $P$ common to two different convex subsets $A,B\subset \R^{d+1}$. A  lemma of Ewald-Larman-Rogers \cite{ELR} states that the latter set has the expected Hausdorff dimension; this is the key lemma for their well-known theorem stating that the set of directions of line segments lying in the boundary of a given convex body in $\Rd$ has Hausdorff dimension at most $d-2$. As shown in \cite{PR}, this argument is enough to show that the set of tangent hyperplanes $P$ to the graph of the DC function $f$ has the expected Hausdorff dimension $d$, which in turn is enough to show that a WDC set admits a normal cycle in the sense of \cite{fu94}, Theorem 3.2. 

However, necessary for our main result (Theorem B below) is the stronger assertion that the set of pairs $(x,P)$ such that $P$ is a tangent plane for the graph of $f$ at $(x,f(x))$ has the same  dimension $d$, and moreover that this dimension may be evaluated in the sense of Minkowski content. This result (Lemma \ref{duality})  follows from a refinement of the Pavlica-Zaj\'i\v cek result: given convex $A,B\subset \Rd$, the set of pairs $(x-y,P)$ such that $P$ is a support plane both for $A$ at $x$ and for $B$ at $y$ has  Minkowski dimension $d$. 
%We establish this refinement in Corollary \ref{noncompact}. 
%Moreover, the original conclusion of \cite{ELR} is stated in terms of Hausdorff dimension, which may behave badly under Cartesian products (cf. \cite{gmt}, \S2.10.29); in order to apply the formalism of \cite{fu90, fu94} to obtain kinematic formulas it is necessary to  show that a Cartesian product $\nor f\times \nor g$ (with $g$ an aura for a second \WDC~ set $B$) is small enough. 
%In fact, examination of the original argument of \cite{ELR} reveals that the conclusion may be stated in terms of {\it Minkowski dimension} without any further work.
 As a byproduct of this analysis we also arrive at the following enhanced version of the theorem of Ewald-Larman-Rogers:
% Thus as a side conclusion we can give also a version of the main theorem of \cite{ELR} that is stronger in two different ways:  given convex body $A\subset \Rd$, the set of pairs $(\sigma,P)$, such that $\sigma\in S^{d-1}$ is parallel to a line segment $\sigma' \subset\partial A$ and $P$ is  a support hyperplane for $A$ at every $x\in \sigma'$, has $\sigma$-finite $(d-2)$-dimensional Minkowski content (Corollary \ref{cor:sharp elr}):
% This discussion implies the following sharpening of the main theorem of \cite{ELR}:
\begin{thmA} Let $K \subset \Rd$ be closed and convex. Denote by $T_K$ the set of pairs
$
(v,w) \in S^{d-1}\times S^{d-1}$ with the property that there exists a nondegenerate segment $\tau \subset \partial K$ with direction $v$ and lying in a supporting hyperplane of $K$ with outward normal direction $w$. Then $T_K$ has $\sigma$-finite $(d-2)$-dimensional Minkowski content.
\end{thmA}

In Section \ref{sect:kinematic} we prove our main theorem, that the kinematic formulas described in \cite{fu90} hold for pairs of WDC sets in an isotropic space $(M,G)$ (classical work of Hartman \cite{hartman} implies that the space of DC functions is stabilized by diffeomorphisms, hence this notion, and the notion of WDC set, make sense on a smooth manifold). In order to state this precisely, let $d$ be the dimension of $M$. Since $(M,G)$ is Riemannian isotropic, it is clear that the space of $G$-invariant differential forms on $SM$ is canonically isomorphic to the subspace of $G_{\bar o}$-invariant elements of the exterior algebra $\bigwedge^*T_{\bar o} (SM)$, where $G_{\bar o} \subset G$ is the stabilizer of an arbitrary point  $\bar o\in SM$. In particular, this space has finite dimension.  Each such form $\beta$ of degree $d-1$ gives rise to a $G$-invariant {\bf curvature measure} on $M$, i.e. an object that associates to each WDC set $A \subset M$ the signed measure $\Phi_\beta(A,\cdot)$ given by
$$
\Phi_\beta(A,E):= \int_{N(A)\with \pi\inv( E)} \beta,
$$
where $\pi:SM \to M$ is the projection. This measure may alternatively be viewed as a linear functional on the space of bounded Borel measurable functions $\phi$ given by 
$$
\Phi_\beta(A,\phi):= \int_{N(A)}\pi^*\phi\wedge \beta.
$$
Denote the space of all such $\Phi_\beta$ by $\C^G$.  For $\Phi, \Psi \in \C^G$ and $A,B \in \WDC(M)$ we put 
$$
(\Phi\otimes \Psi)(A,\phi;B,\psi) := \Phi(A,\phi) \Psi(B,\psi)
$$
and extend to all of $\C^G \otimes \C^G$ by bilinearity.

\begin{thmB}\label{thm:main}  Let $(M,G)$ be a Riemannian isotropic space, and put $d\gamma$ for the Haar measure on $G$ that projects to the Riemannian volume of $M$.
\begin{enumerate}
\item  If $A,B \in \WDC(M)$ then $A \cap \gamma B \in \WDC(M)$ for a.e. $\gamma \in G$.
\item \label{item:2} There exists a linear map 
$$
K: \C^G \to \C^G \otimes \C^G
$$
such that, for any compact $A,B\in \WDC(M)$ and bounded Borel measurable functions $\phi,\psi: M\to \R$ 
\begin{equation}\label{eq:main}
\int_G \Phi(A \cap \gamma B, \phi \cdot(\psi \circ \gamma\inv)) \, d\gamma = K(\Phi)(A,\phi;B,\psi).
\end{equation}
\end{enumerate}
\end{thmB}

Conclusion \eqref{item:2}  may be restated in more prosaic terms as follows:

{\sl Let $\beta_1,\dots, \beta_N$ be a basis for the vector space of $G$-invariant differential forms of degree $d-1$ on $SM$. Then there exist constants $c_{ij}^k$ such that given any compact sets $ A,B \in \WDC(M)$ and bounded Borel measurable functions $\phi,\psi:M \to \R$, then
\begin{align}
\notag\int_G \left(\int_{N(A\cap \gamma \inv B)} \pi^*(\phi \cdot (\psi \circ \gamma\inv))\wedge \beta_k \right)\, d\gamma &= \sum_{i,j} c^k_{ij} \int_{N(A)} \pi^*\phi \wedge \beta_i \int_{N(B)}\pi^*\psi\wedge \beta_j \\
\label{eq:main kf}& +\int_A \phi \cdot \int_{N(B)} \pi^*\psi \wedge \beta_k \\
\notag & + \int_B\psi \cdot \int_{N(A)} \pi^*\phi\wedge \beta_k , \    k=1,\dots,N.
\end{align}
}

For $A,B$ of positive reach and $(M,G) = (\Rn ,\barson)$ this is the classical kinematic formula of Federer \cite[Theorem~6.11]{cm}. 

\subsection{Acknowledgments} It is a pleasure to thank L. Zaj\'i\v cek for helpful conversations.

\section{Background}\label{sect:generalities}

\subsection{Generalities and notation}\label{sect:cosphere} 
\subsubsection{General concepts}
The volume of the unit ball in $\R^d$ is denoted 
$$
\omega_d:=\pi^{\frac d2}/\Gamma\left(\frac d2+1\right).
$$
We put $\Pi_L$ for the orthogonal projection onto an affine subspace $L\subset \Rd$.

We write $A \csubset B$ to mean that $A$ is a compact subset of $B$.

If $E$ is a Cartesian product $ A\times B$ or $B\times A$,  then we use $\pi_A: E \to A$ to denote the projection.

%need to distinguish between $\nor^*(f,0)$, $\nor(A,x)$

%We put for functions $f,g$ defined on a common space
%$$
%f\vee g := \max (f,g).
%$$

\subsubsection{Currents} We generally follow the notation and terminology of \cite{gmt}, with some minor deviations. A {\bf current} $T$ of dimension $k$ in the smooth manifold $M$ is a linear functional $T$ on the space $\Omega_c^k(M)$ of compactly supported smooth differential forms on $M$, continuous with respect to $C^\infty$ convergence with uniformly compact support. 
The {\bf support} of $T$ is denoted $\spt T$.
The pairing of a current $T$ against a differential form $\phi$ will usually be denoted $\int_T \phi$. A current $T$ is {\bf representable by integration} if there exist a Radon measure $\lVert T \rVert$ and a  Borel measurable $k$-vector field $\vec T$ on $M$ such that 
$$
\int_T \phi = \int_M \langle \vec T_x, \phi_x\rangle \, d\lVert T \rVert x, \quad \phi \in \Omega_c^k(M).
$$

If $M$ is an oriented $C^1$ manifold, we will often conflate a measurable subset $A \subset M$ with the current defined by integration over $A$.

Among all currents the group 
$\mathbb I_k(M)$  of {\bf integral currents} (\cite{gmt}, 4.1.24) of dimension $k$ in manifold $M$ enjoys many special properties.
 Any integral current $T$ may be pushed forward by a proper Lipschitz map $f:M \to N$ to yield a current $f_*T \in \mathbb I_k(N)$ via the formula
$$
\int_{f_*T}\phi := \int_T f^*\phi.
$$

We will rely heavily on the Federer-Fleming theory of {\bf slicing}, described in detail in Section 4.3 of \cite{gmt}. The slice of $T$ by such $f$ at $y\in N$ is denoted $\langle T,f,y\rangle$. If $T$ is given by integration over a smooth submanifold $V$, and $y$ is a regular value of $f$, then this slice is simply the current given by integration over the appropriately oriented intersection of $V$ with $f\inv(y)$. If $S \in \I_k(Y)$ for some manifold $Y$, recall that the conventions of \cite{gmt}, Section 4.3, imply that for $y\in N$
\begin{align}\label{eq:slice orientation} 
\pi_{Y*} \langle N\times S, \pi_N, y\rangle & =  S,\\
\notag \pi_{Y*} \langle S\times N, \pi_N, y\rangle &=  (-1)^{k \dim N} S,
\end{align}
where the Cartesian product of currents is defined as in \cite{gmt}, 4.1.8.
One particularly important fact is {\bf commutativity of pushforward and slicing} (\cite{gmt}, Theorem 4.3.2 (7)):
if 
$$
\begin{CD}
M @>g>> N @>h>> L
\end{CD}
$$
are Lipschitz maps, and $T \in \mathbb I_*(M)$ then
 for a.e. $y \in L$ 
\begin{equation}\label{eq:funct of slice}
 g_* \langle T, h\circ g , y \rangle =  \langle  g_*T, h , y \rangle.
\end{equation}

The boundary of a current $T$ is denoted by $\partial T$, and defined by the formula
$$
\int_{\partial T}\phi:= \int_{T}d\phi,
$$
and the boundary of a Cartesian product is
$$
\partial (S\times T) = \partial S \times T + (-1)^{\dim S} S \times \partial T.
$$
Slicing behaves naturally with respect to the boundary operation: if $g:M\to N, \dim N = n$ then for a.e. $y \in N$ (cf. \cite{gmt}, p. 437)
\begin{equation}\label{eq:slice boundary}
\partial \langle T, g, y\rangle = (-1)^n  \langle \partial T, g, y\rangle .
\end{equation}

If $T$ is a current of dimension $k$ and $\phi$ is a differential form of degree $j\le k$ then $T\with \phi$ is the current of dimension $k-j$ defined by
$$
\int_{T\with \phi}\psi= (T\with \phi)(\psi) := \int_T \phi\wedge \psi.
$$
If $T$ is integral (in particular, representable by integration), the  formula above makes sense even when $\phi$ is merely bounded and Borel measurable. In particular, if $A\subset M$ is a Borel subset we set
$$ \int_{T\with A}\psi = (T\with A)(\psi):=\int_T (\psi\cdot 1_A),$$
where $1_A$ denotes the characteristic function of $A$.

\subsubsection{Manifolds and bundles}\label{sect:manifold conventions} If $M$ is a smooth manifold then we denote by $S^*M$ its cosphere bundle. The elements of the total  space  of $S^*M$ may be thought of either as rays lying in the fibers of the cotangent bundle with endpoint at $0$, or else as oriented hyperplanes through the origin within the tangent spaces $T_xM$. 
For convenience we will sometimes make use of an arbitrarily chosen $C^1$ $1$-homogeneous length function $\ell:T^*M \to [0,\infty)$,  positive off of the zero section (for example, one induced by a Riemannian metric on $M$). In this case we may also think of $S^*M$ as the subspace  $\ell\inv(1) \subset T^*M$.
We put 
$$
\proj:T^*M \setminus \z \to S^*M
$$ 
for the canonical projection.
If $M$ is Riemannian then $SM \subset TM$ is the bundle of unit tangent vectors.
 Abusing notation, we put $\pi$ for the projection of any of the bundles  $TM$, $T^*M$, $SM$, $S^*M$ to $M$, and $\proj:TM \setminus \z \to SM$ for the normalization map.
 
 We will frequently consider the case of a homogeneous space $M = G/G_o$, where $G$ is a finite dimensional Lie group and $G_o$ is the stabilizer of a base point $o\in M$. We orient $G,G_o,M$ so that whenever $M\supset U \owns x\mapsto \omega_x\in G$ is a smooth local section, $\omega_xo = x$, the map 
\begin{equation}\label{eq:orientation convention}
U \times G_o \owns (x,\bar \gamma)\mapsto \bar \gamma \omega_x \inv
\end{equation}
is an orientation-preserving diffeomorphism onto the corresponding open subset of $G$--- there exist four different choices of such a system of orientations, but the distinctions among them will be immaterial.
We denote by $\mathcal F$ the bundle over $M \times M$ with fiber $G_o$ and total space
 $$
\FFF:= M\times M\times_{\mathcal F} G_o:= \{(x,y,\gamma) \in M\times M\times G: \gamma y = x\} .
 $$
We orient these spaces consistently with the local product structure, i.e. if $\omega_x,  \omega_y$ are local sections  as above defined on open subsets $U,V$ then
\begin{equation}\label{eq:bundle orientation}
(x,y, \bar \gamma) \mapsto (x,y, \omega_x \bar \gamma \omega_y\inv)
\end{equation}
yields an orientation-preserving local diffeomorphism  $U \times V \times G_o \to \FFF$. One may easily check that with these conventions the diffeomorphism
\begin{equation}\label{eq:orientation of F}
M \times G\to M\times M\times_{\mathcal F} G_o, \quad (x,\gamma) \mapsto (x, \gamma\inv x, \gamma)
\end{equation}
also preserves orientations. Abusing notation, we will use the same notation to denote pullbacks of $\F$ by maps into $M\times M$, e.g. 
 $$
 TM \times TM \times_{\F} G_o:= \{(\xi,\eta,\gamma) \in TM\times TM\times G: \gamma \pi(\eta) = \pi(\xi)\} ,
 $$
 or, if $S,T$ are currents living in $TM$, then $S \times T \times_{\F} G_o$ is the current given in an orientation-preserving local trivialization as the Cartesian product.

% \subsection{Orientations}\label{sect:orientation} 
 
 We will put $\Gamma$ for the projection of $\FFF\subset M\times M \times G$ to the third factor, and $X,Y$ the projections to the  respective $M$ factors. We will abuse notation by using the same symbols also to denote the maps on pullbacks of $\mathcal F$ obtained by precomposing with the associated maps into $\FFF$.

By the orientation conventions \eqref{eq:slice orientation} and \eqref{eq:orientation of F},
\begin{equation}\label{eq:Pi1}
X_*\langle \FFF, \Gamma, \gamma \rangle = (-1)^{d\dim G} M.
\end{equation}
Define the involutions 
\begin{align*}
\iota:\FFF \to \FFF, \ &(x,y,\gamma) \mapsto (y,x,\gamma\inv) , \\
I:G \to G,  \  &\gamma \mapsto \gamma \inv,
\end{align*}
both of parity $(-1)^{\dim G } = (-1)^{d + \dim G_o}$.
 Conclusions (6) and (7) of \cite{gmt}, Theorem 4.3.2 now yield
\begin{align}
\notag Y_*\langle \FFF, \Gamma, \gamma \rangle &= (-1)^{\dim G} Y_*\langle \FFF,I\circ \Gamma, \gamma\inv \rangle \\
\notag &= (-1)^{\dim G} Y_*\langle \FFF, \Gamma\circ \iota, \gamma\inv \rangle \\
\label{eq:Pi2}  &= (-1)^{\dim G} Y_*\iota_*\langle \iota_* \FFF, \Gamma, \gamma\inv \rangle \\
\notag    &=  (Y\circ\iota)_*\langle  \FFF, \Gamma, \gamma\inv \rangle \\
\notag        &=  X_*\langle  \FFF, \Gamma, \gamma\inv \rangle \\
\notag        &= (-1)^{d\dim G} M.
\end{align}
%%&= (-1)^{\dim G} Y_*\langle F, I\circ \Gamma, \gamma^\inv \rangle \\
%%&= (-1)^{\dim G}X_* \langle \iota_*F,  \Gamma\circ \iota, \gamma \rangle \\
%%(X \circ \iota)_* \langle F, \Gamma, \gamma \inv\rangle \\
%%\label{eq:Pi2}&= X_* \langle \iota_*F, \Gamma, \gamma\inv \rangle \\
%%\label{eq:Pi2*}&= 
%%\notag&= (-1)^{d+\dim G}  X_*\langle F, \Gamma, \gamma \rangle\\
%%\notag&= (-1)^{d+\dim G} M.

 If $M$ is Riemannian and $G$ acts on $M$ by isometries then
we may endow $G$ with an invariant volume form $d\gamma$ compatible with the given orientation, such that the corresponding volume measure on $G$ projects to the positively oriented Riemannian volume form $d\vol_M$ of $M$ (i.e. $d\gamma$ is the product of the Riemannian volume of $M$ with the invariant probability volume form on the compact subgroup $G_o$). Let $\pi_{G_o*}: \Omega^*(\FFF) \to \Omega^*(M\times M)$ denote  fiber integration over $G_o$.  Since the maps \eqref{eq:orientation convention}, \eqref{eq:bundle orientation} preserve orientation, we observe that 
\begin{align}
\label{eq:pi Go 1}\pi_{G_o*}(\Gamma^* d\gamma )&\equiv  Y^* (d\vol_M) \mod X^*\Omega^*(M),\\
\label{eq:pi Go 2}\pi_{G_o*}(\Gamma^* d\gamma )&\equiv  (-1)^d X^* (d\vol_M) \mod Y^*\Omega^*(M).
\end{align}

\subsubsection{Convexity}\label{sect:convex} By a {\it convex body} we understand a non-empty, compact and convex subset of $\rd$. If $K$ is a convex body and $n\in S^{d-1}$ a unit vector, the {\it support function} of $K$ at $n$ is
$$h_K(n)=\sup\{ x\cdot n:\, x\in K\}.$$
For $t>0$  we denote by
$$
C(K,n,t):=\{x\in K:x\cdot n\geq h_K(n)-t\}
$$
the {\it cap} of $K$ of direction $n$ and width $t$.
If $x\in\partial K$ we write
 $$
 \nor (K,x)
 $$ 
 for the set of all unit outer normal vectors to $K$ at $x$ (these are vectors from the dual cone to the tangent cone of $K$ at $x$). The {\it width} of $K$ is defined as
$$\wth K=\inf\{ h_K(n)+h_K(-n):\, n\in S^{d-1}\}.$$

The symbol $\Delta$ will denote the difference operator on sets:
$$\Delta A:=A-A=\{ a-b:\, a,b\in A\}.$$
%$V_d$ is the volume of $d$-dimensional unit ball, 

\subsection{Minkowski content}\label{sect:minkowski}
\begin{definition} \rm
%Let $M$ be a $C^1$ Riemannian manifold of dimension $d$ and  $m \ge 0$.  
The {\it $m$-dimensional upper Minkowski content} of  $S\subset \Rd$ is
$$
{\cal M}^{*m}(S)=\limsup_{\eps\downarrow 0} (2\eps)^{m-d}\vol(S_\eps),
$$
%where $\vol$ denotes the Riemannian volume in $M$ and 
where $S_\eps$ is the set of points in $M$ lying within distance $\eps$ of $S$. If $\cal M^{*m}(S) <\infty$ then we say that $S$ has {\bf finite $m$-content}.
%We further say that $S$ has {\bf locally finite $m$-content } if  $S = \bigcup_{i=1}^\infty S_i$ where each $S_i$ has finite $m$-dimensional upper Minkowski content.
% This notion is clearly independent of the choice of Riemannian metric on $M$.
\end{definition}

%\begin{rem}  \rm
%\begin{enumerate}
%\item The normalization in the definition of the Minkowski content is not unique in the literature, but does not matter in our case since we are interested only in the finitness, and not in its particular value. Our definition follows Mattila \cite[\S5.5]{Mat95} (considering the Euclidean space).
%\item Note that a change of the Riemannian metric would not affect locally the finitness of the upper Minkowski content. Hence, if $M$ is a compact $C^1$ manifold, the property of $\sigma$-finite $m$-dimensional content can be considered independently of a chosen Riemannian metric.
%\end{enumerate}
%\end{rem}

For $S$ as above and $\eps>0$, we define the $\eps$-{\it covering number} of $S$
$$\#(S,\eps)=\min\left\{ k:\, S\subset\bigcup_{i=1}^k B(x_i,\eps)\text{ for some }x_1,\dots, x_k\in\R^d\right\}.$$
 
\begin{lem}\label{lem:box} \
\begin{enumerate}
\item $S$ has finite $m$-content iff
$$ \limsup_{\eps\downarrow 0} \eps^{-m}\#(S,\eps)<\infty.$$
\item If $S$ has $\sigma$-finite $m$-content then $S$ has 
%both
 Hausdorff
%  and packing 
  dimension $ \le m$.
\item \label{item:mink prod} If $S$ has finite $m$-content and $T$  has finite $n$-content, then $S \times T$ has finite $(m+n)$-content.
\end{enumerate}
\end{lem}
\begin{proof} 
(1) follows from the inequalities
$$P(S,\eps)\,\omega_d\eps^d\leq\vol(S_\eps)\leq \#(S,2\eps)\,\omega_d(2\eps)^d,$$
where $P(S,\eps)$ is the maximal number of disjoint $\eps$-balls with centres in $S$ ($\eps$-packing number of $S$), and from 
$$\#(S,2\eps)\leq P(S,\eps)\leq \#(S,\eps/2),$$
see \cite[\S5.3-5.5]{Mat95} for details. Assertions (2) and (3) follow at once from (1).
% by standard techniques, cf.\ \cite[Chapter~5]{Mat95}.
\end{proof}

\subsection{Lipschitz and \DC~ functions}\label{sect:ma and dc} 
%\subsubsection{Generalities about Lipschitz functions} 
If $f$ is a locally Lipschitz function defined on a $d$-dimensional $C^1$ manifold $M$,  we denote by $\setd f(x)$ its {\bf Clarke differential} \cite{clarke} at  $x\in M$. To be explicit, we take
$
\setd f(x) \subset T^*_xM$ to be the convex hull of the set of all $\xi \in T^*_xM$  with the following property: there exists a sequence $ x_1,x_2,\dots \to x$, such that $ f$ is differentiable at each $ x_i$, and $\lim_{i\to\infty}df(x_i)= \xi $. 
Since by Rademacher's theorem such $f$ is differentiable a.e., this defines a nonempty compact convex subset of $T^*_xM$. 
If $\psi$ is $C^1$ then the chain rule
$$
\setd (f\circ \psi) (x) \subset \psi^* (\setd f(\psi(x)))
$$
holds (cf. \cite{clarke} for this and other basic relations regarding the Clarke differential).
%
%The point $x$ is a {\bf critical point} of $f$ if $o\in \partial f(x)$, and otherwise a {\bf regular point}. If $x$ is a regular point theIf $0\notin \setd f(x)$ then the implicit function theorem of Clarke (\cite{clarke}, Section 7.1) states that there exists a local coordinate patch $\psi: U \to \Rd$ such that $\psi(f\inv(f(x)\cap U))$ is the graph of a Lipschitz function defined on some open subset $V \subset \R^{d-1}$.

\begin{lem}\label{lem:local extremum}
If $x$ is a local extremum of $f$ then $0\in \setd f(x)$. \quad $\square$
\end{lem}

\begin{lem}[\cite{clarke}, Proposition 2.3.3
% and 2.3.12
 ] \label{lem:clarke of max} If $f,g:M\to \R$ are locally Lipschitz functions  then
$$
\partial (f+g)(x) \subset \partial f(x) + \partial g(x),
$$
where $+$ on the right hand side denotes Minkowski sum.
%, and
%\begin{equation}\label{eq:clarke max}
%\partial (f\vee g)(x) \begin{cases} = \partial f(x) &  f(x) >g(x)\\
%=\partial g(x) & g(x) >f(x) \\
%\subset \conv (\partial f(x) \cup \partial g(x)) & f(x) = g(x).
%\end{cases}
%\end{equation}
\end{lem}

\begin{definition}\label{def:weakly reg}{\rm Put  
$$
 \graph (\partial f) := \{\xi \in T^*M: \xi \in \partial f(\pi(\xi))\}.
 $$
 Clearly $\graph (\partial f)$ is a closed subset of $T^*M$.
 
 A number $c \in \R$ is a {\bf weakly regular value} of $f$ if 
 \begin{equation}
\clos\left(\graph (\partial f) \cap \pi\inv f\inv((c,\infty))\right) \cap \pi\inv f\inv(c) \cap \z = \emptyset.
\end{equation}
This condition is equivalent to each of the following statements:
\begin{enumerate} \item Whenever $M\owns x_1,x_2,\dots \to x_0$, with $f(x_i )> f(x_0) = c$,  and $\xi_i \in \partial f(x_i)$, then $\xi_i\not \to 0$.
\item Let $\ell: T^*M \to [0,\infty)$ be a   length function as above. Then for any $K\csubset M$ there exists $\eps >0$ such that  
$$
x\in K, \ v\in \partial f(x), \ c<f(x)< c+ \eps  \implies \ell(v) \ge \eps.
$$
\end{enumerate}}
\end{definition}

%Note that, again, the obvious chain rule applies to $\partial_+(f\circ \psi)$ for $f$ Lipschitz and $\psi \in C^1$.

%\subsubsection{Differences of convex functions} 
\begin{definition}{\rm
A function $f$ defined on an open set $U\subset \Rd$ is  {\bf\DC~} if for every $x\in U$ there is some convex neighborhood $V\subset U$ of $x$, and convex functions $g,h:V \to \R$, such that $f=g-h$ on $V$.  The class of all such functions is denoted $\DC(U)$.}
\end{definition}
Obviously every \DC~ function is locally Lipschitz. 
The class of \DC~ functions enjoys many remarkable properties, prominently the following classical result of Hartman.

\begin{thm}[\cite{hartman}]\label{thm:hartman} Let $U\subset \Rn, V \subset \Rd$ be open. Suppose $\psi=(\psi_1,\dots,\psi_n): V \to U$, where the $\psi_i \in \DC(V)$, and $f \in \DC(U)$.  
Then $f\circ\psi \in \DC(V)$. 
\end{thm}

Thus  if $M$ is a $C^{1,1}$ manifold we may define $\DC(M)$ to be the space of all functions $f:M\to \R$ with the property that $f\circ \psi\inv \in \DC(U)$ whenever $(\psi,U)$ is a $C^{1,1}$ coordinate patch for $M$.

\begin{cor}\label{cor:max} If $f,g\in \DC(M)$ then $f+g, f\vee g:= \max (f,g) \in \DC(M)$.
\end{cor}

%The Clarke implicit function theorem admits a \DC~ version:
%
%\begin{thm}[Vesel\'y and Zaj\'i\v cek] \label{thm:dc it} Suppose $f \in \DC(M)$, and suppose that $c \in \R$ is a regular value of $f$ in the sense that $0\notin \partial f(x)$ for all $x \in f\inv(c)$. Then for each $x \in f\inv(c)$ there exists a coordinate patch $\psi:U\to \Rd$ for $M$ about $x$ such that $\psi(U \cap f\inv(c))$ is the graph of a \DC~ function defined on an open subset of $\R^{d-1}$.
%\end{thm}
%\begin{proof} Taking local coordinates, we apply the implicit function Theorem 4.4 of \cite{VZ}.
%\end{proof}

\subsection{\WDC~ sets}

The definition of these objects is motivated by the following. Recall that a function $f$ defined on an open subset of $\Rd$ is {\bf semiconvex} if it may be expressed locally as the sum of a convex function and a smooth function. It is clear that any semiconvex function, and in particular any $C^{1,1}$ function, is \DC.

\begin{thm}[Kleinjohann \cite{kleinjohann}, Bangert \cite{bangert}]\label{thm:klein-bang}  A set $A \subset U\subset \Rd$ has locally positive reach iff  $A = f\inv(-\infty,0]$, where $f:U \to \R $ is a  semiconvex function and $0$ is a weakly regular value of $f$. 
\end{thm}

\begin{definition} {\rm Let $M$ be a  a $C^2$ manifold.  A subset  $A\subset M$ is a {\bf \WDC~ subset} of $M$ (or simply a {\bf  \WDC~ set}) if $A = f\inv(-\infty,c]$ for some $f \in \DC(M)$ and some weakly regular value $c$ of $f$.

If $c=0$ and $f\ge 0$ then $f$ is a {\bf \DC~ aura} (or simply an {\bf aura}) for $A$. }
\end{definition}

\noindent{\bf Remark.} This terminology is different from that of \cite{fu94}, in which the function $f$ would be referred to as a {\it nondegenerate} aura. The point there was that the weak regularity condition may be removed if the function involved is subanalytic. In the present paper, however, all auras will be nondegenerate.

\begin{prop} Every \WDC~ set admits a DC aura.
\end{prop}
\begin{proof} Given $A,f,c$ as above, Corollary \ref{cor:max} implies that $(f-c)\vee 0$ is a DC aura for $A$.
\end{proof}

\begin{definition} \rm Let $M$ be a $C^2$ manifold and $f$ an aura for a \WDC~ set $A \subset M$.
 Given $\eps >0$ we denote
$$
\noreps f:= \nu(\{v: \ell(v) \ge \eps, \text{ and }v\in \partial f(x) \text{ for some } x \in \bdry A\}) \subset S^*M.
$$
\end{definition}

Since the graph of the Clarke differential $\partial f$ is closed, it follows that $\noreps(f)$ is a closed subset of $S^*M$ for every $\eps >0$.

\subsection{ Monge-Amp\`ere functions} Let $M$ be an oriented $C^2$ manifold of dimension $d$. The cotangent bundle $T^*M$ carries a natural {\bf canonical 1-form} $\alpha \in \Omega^1(T^*M)$ given by
$$
\langle \alpha_\xi, \tau \rangle := \langle \xi, \pi_* \tau\rangle, \quad \xi \in T^*M, \tau \in T_\xi T^*M.
$$
The exterior derivative $\omega:= d\alpha \in \Omega^2(T^*M)$ is the standard symplectic form of $T^*M$.

We recall that $f \in W^{1,1}_{loc}(M)$ is said to be {\bf Monge-Amp\`ere} (or {\bf MA}) if there exists an integral current $\D(f) \in \mathbb I_d(T^*M)$ satisfying the axioms of  \cite{fuMA, jerrard1, jerrard2}, i.e.
\begin{enumerate}
\item $\partial \D(f)=0$;
\item $\D(f)$ is Lagrangian, i.e. $\D(f) \with \omega = 0$;
\item $\operatorname{mass}(\D(f)\with \pi\inv (K))<\infty$ for every $K\csubset M$ (the mass may be computed with respect to the Sasaki metric corresponding to any $C^2$ Riemannian metric on $M$);
\item \label{item:integral condition} for any smooth  volume form $d\vol_M \in \Omega^d(M)$ and  every $g\in C^{\infty}_c(T^*M)$,
$$
\int_{\D(f)}g \wedge \pi^*(d\vol_M)=\int_M g(x,df(x))\: d\vol_M.
$$
\end{enumerate}
By Theorem 4.3.2(1) of \cite{gmt}, the condition \eqref{item:integral condition} may be replaced by the  equivalent condition

\indent \quad($4'$) $\langle \D(f),\pi, x\rangle  = \delta_{(x,df(x))}$ for a.e. $x \in M$.

%\noindent or the  inelegant but useful characterization ($4"$) below. In order to state this we need a simple Lemma.
%
%\begin{lem} Let $E$ be the total space of a smooth vector bundle over a smooth oriented manifold $M$ with volume form $d\vol_M$. Suppose $\sigma$ is a Borel measurable section of $E$. Then there exists at most one integer multiplicity rectifiable current $G$ of dimension $d$ in $E$ such that
%\begin{enumerate}
%\item $\pi_* G = M$
%\item $\pi_* \lVert G \rVert$ is absolutely continuous with respect to $d\vol_M$, and 
%\item $G$ is carried by the graph of $\sigma$.
%\end{enumerate}
%\end{lem}
%
%
%If it exists, we denote this current $G$ by $\graph (\sigma)$. It is clear that if $\sigma = \sigma'$ a.e. then $\graph(\sigma) = \graph(\sigma')$.
%
%\begin{proof} By conditions (2) and (3), results from \cite{gmt} imply that the measure $\lVert G \rVert$ may be taken 
%as the pushforward measure $\sigma_*(d\vol_M)$. Since $G$ is rectifiable, there exists $F \subset M$ such that $\sigma(M\setminus F)$ is a rectifiable subset of $E$ of dimension $d$.
%\end{proof}
%
%We now state the final equivalent form of the fourth axiom above:
%
%\indent ($4''$) $\graph (df)$ exists, and
%$$
%(\D(f) - \graph(df)) \with \pi^* d\vol_M = 0.
%$$
%

As shown in the papers cited above, these axioms determine $\D(f)$ uniquely if it exists.  We denote the class of all such $f$ by $\MA(M)$. Strictly speaking, the discussion there applies to the case where $M$ is an open subset of $\rd$; starting from that formulation the class $\MA(M)$ may also be defined in terms of local coordinates, in view of the following.

\begin{lem}\label{lem:diffeo invariance} 
Let $U,V \subset \rd$ be open, and $\psi:U\to V$ a $C^2$ diffeomorphism. Put $\tilde \psi = (\psi\inv)^*:T^*U \to T^*V$ for the induced $C^1$ diffeomorphism of cotangent bundles. If $f \in \MA(V)$ then $f \circ \psi \in \MA(U)$, with
$$
\D(f\circ \psi) = \tilde \psi_*\D(f).
$$
\end{lem}
\begin{proof}
Using the fact that $\tilde \psi$ is a symplectomorphism, it is easy to confirm that the right hand side satisfies the axioms above with $f$ replaced by $f \circ \psi$.
\end{proof}

%If $M$ is an oriented $C^2 $ manifold of dimension $d$ then we  
%define the class $\MA(M)$ of {\it Monge-Amp\`ere (\MA) functions}  on $M$ as the class of all functions $f\in W^{1,1}_{loc}(M)$ with the following property: there exists a $C^2$ atlas $(\psi_\alpha, U_\alpha)$ of $M$ such $f\circ \psi_\alpha\inv \in \MA(\psi(U_\alpha))$ for all $\alpha$. Lemma \ref{lem:diffeo invariance} implies that this property holds for any such atlas, and that we may define the {\bf differential cycle} $\D(f)\in \mathbb I_d(T^*M)$ of $f$ by
%$$
%\D(f) \with \pi\inv (U_\alpha) = \tilde \psi_{\alpha *}\inv \D(f\circ \psi_\alpha\inv)
%$$
%where $\tilde \psi_\alpha: \restrict{T^*M}{U_\alpha} \to T^*(\psi_\alpha(U_\alpha))$ is the induced map. Alternatively, one may reformulate axioms (1)-(4), ($4')$, ($4''$) direct

The  starting point for the main constructions of this paper is the following.

\begin{thm}[\cite{PR}]\label{thm:dc is ma} 
Every \DC~ function is \MA.$ \quad \square$
\end{thm}

We will also need the following fundamental fact.

\begin{lem}[\cite{fuMA}]\label{lem:support df} If $f$ is a locally Lipschitz \MA~ function then
$$
\spt \D(f) \subset \graph \partial f. \quad \square
$$
\end{lem}

\subsubsection{Sums of MA functions on a homogeneous space} Although the class $\MA(M)$ is not closed under addition, it is closed under addition in general position in a sense given in the next Theorem, a variant of Proposition 2.6 of \cite{fuMA}, part I.

 Let $G$ be a Lie group and $M= G/G_o$ an oriented homogeneous space of $G$, where $G_o$ is the stabilizer of the arbitrarily chosen base point $o \in M$. 
%We assume that $G_o$ is compact. 
Abusing notation, we denote simply by $\gamma $  the induced action of $\gamma \in G$ on $T^*M$ by symplectomorphisms, i.e.
pullback under the diffeomorphism $\gamma\inv:M \to M$:
\begin{equation}\label{eq:def lift gamma}
 \gamma \xi := (\gamma\inv)^* \xi
\end{equation}
Referring to the convention of Section \ref{sect:manifold conventions}, consider the smooth manifold
$$
\tilde \FFF:=
T^*M \times T^*M \times_{\F} G_o.
$$
 We put
$ \Gamma :\tilde \FFF \to G $
for the map given by the restricted projection to $G$,  and
$$
\Sigma:\tilde \FFF \to T^*M, \quad \Sigma(\xi,\eta,\gamma): = \xi +\gamma \eta.
$$
We put also $X,Y:\tilde \FFF \to M$ for the respective projections to the base spaces of the first and second factors.

\begin{thm}\label{thm:ma addition} Suppose the manifold $M$ is an oriented homogeneous space of $G$, as above, and let $f,g\in \MA(M)$. Then $h_\gamma:=f + g\circ \gamma\inv \in \MA(M)$ for a.e. $\gamma\in G$, with 
\begin{equation}\label{eq:diff of sum}
\D(h_\gamma) =  (-1)^{d\dim G} \Sigma_*\langle \D(f) \times \D(g) \times_{\F} G_o, \Gamma, \gamma\rangle. 
\end{equation}
\end{thm}
\begin{proof} We  check that for a.e. $\gamma \in G$  the right hand side satisfies the axioms (1), (2), (3), ($4'$) for \MA~ functions, with $f$ replaced by $f + g\circ\gamma \inv$. 

Axiom (1) is immediate from \eqref{eq:slice boundary}.

To show (2), by \cite{gmt}, Theorem 4.3.2(1) it is enough to prove the following claim. Let $\Xi, H: \tilde \FFF \to T^*M$ denote the restrictions to $\tilde \FFF\subset T^*M \times T^*M \times G_o$ of the projections to the first and second factors, respectively. Then
\begin{equation}\label{eq:omega = omega}
 \Sigma^* \omega \equiv 2\Xi^* \omega + 2H^*\omega \mod \Gamma^*\Omega^*(G).
\end{equation}
In fact we prove the stronger claim that
\begin{equation}\label{eq:alpha = alpha}
 \Sigma^* \alpha \equiv 2\Xi^* \alpha + 2H^*\alpha \mod \Gamma^*\Omega^*(G),
\end{equation}
from which \eqref{eq:omega = omega} follows by taking the exterior derivative.

 To prove \eqref{eq:alpha = alpha}, observe first that each tangent space 
$$
T_{(\xi,\eta,\gamma)} \tilde \FFF\subset T_\xi T^*M\oplus T_\eta T^*M \oplus T_\gamma G
$$
and that the derivative $\Gamma_*$ of $\Gamma$ equals the projection to the last factor on the right. The kernel of this map is clearly
\begin{equation}\label{eq:def V}
V:= \{(\sigma,\tau,0): \sigma \in T_\xi T^*M, \tau \in T_\eta T^*M, \pi_*\sigma = \pi_*\gamma_* \tau \in T_{\pi(\xi)}M \}
\end{equation}
and it is enough to show that 
$$
\restrict{ \Sigma^*\alpha }V =2 \restrict{\Xi^*\alpha }V+ 2\restrict{H^*\alpha }V.
$$
But for $(\sigma,\tau,0) \in V$
\begin{align*}
\langle  \Sigma^*\alpha, (\sigma,\tau,0)\rangle  &= \langle \alpha,  \Sigma_*(\sigma,\tau,0)\rangle \\
& = \langle \alpha, \sigma+\gamma_*\tau\rangle \\
& =  \langle \xi +\gamma \eta , \pi_*(\sigma +\gamma_*\tau)\rangle \\
& =  2\langle \xi  , \pi_*\sigma\rangle  +2\langle \gamma\eta,\gamma_*\pi_*\tau\rangle \quad {\rm by \eqref{eq:def V}}\\
& =  2\langle \xi  , \pi_*\sigma\rangle  +2\langle \eta,\pi_*\tau\rangle \quad {\rm by \ \eqref{eq:def lift gamma}}\\
&=2\langle  \Xi^* \alpha + H^*\alpha,(\sigma,\tau,0)\rangle
\end{align*}
as claimed.

%
%To prove (3) and ($4'$) we recall the manifold
%$
%\tilde\FFF:= T^*M \times T^*M \times_{\F} G_o
%$ and denote the restricted projections to the bases $M$ of the two respective $T^*M$ factors by $X,Y$, and the restricted projection to $G$ by (again) $\Gamma$. 
%%Clearly $F$ is again the total space of a bundle $\E'$ over $M\times M$ with fibers $G_o$, and $E$ is the total space of a vector bundle over $F$ with fibers $T^*_xM \oplus T^*_yM$.

To prove (3), we wish to show that for a.e. $\gamma \in G$ and any $K\csubset M$
$$
\mass \left(\Sigma_*\langle \D(f) \times\D(g) \times_{\F} G_o,\Gamma,\gamma \rangle\with \pi\inv (K)\right) <\infty.
$$
Since $\Sigma$ is Lipschitz when restricted to the preimage of any compact subset of $\FFF$ under the projection $\tilde \FFF \to \FFF$, the last relation on p. 370 of \cite {gmt}, together with Theorem 4.3.2(2), op. cit., imply that it is enough to show that 
$$
\mass\left( (\D(f) \times \D(g) \times_{\F} G_o)\with (\pi \circ \Sigma)\inv (K) \cap \Gamma\inv(J) \right)<\infty
$$
for $K\csubset M, J \csubset G$. But this last current is supported in $X\inv(K) \cap Y\inv(J\inv K)$, and hence the finiteness of the mass follows from the finiteness axiom (3) for the \MA~ functions $f,g$.

%we observe that for any compact sets $J \subset G$ and $K \subset M$, the set $X\inv(K) \cap \Gamma\inv(J)\subset \FFF$ is compact: for this set is the image of $K\times J$ under the map $(x,\gamma) \mapsto (x,\gamma\inv x,\gamma)$. Using axiom (3) for $\D(f),\D(g)$, this  implies that the mass of $(\D(f) \times \D(g) \times_{\F} G_o)\with (\pi\inv X\inv(K)\cap \Gamma \inv(J))$ is finite. The desired conclusion now follows from \cite{gmt}, Theorem 4.3.2 (2).

Finally we prove $(4')$. By commutativity of pushforward and slicing it is enough to show that 
\begin{equation}\label{eq:sum slice}
(-1)^{d\dim G} \langle\langle \D(f) \times \D(g) \times_{\F} G_o, \Gamma, \gamma\rangle, X, x\rangle
= \delta_{(x,df(x))}\times \delta_{(\gamma\inv x,dg(\gamma\inv x))}\times \delta_\gamma
\end{equation}
for a.e. $(x,\gamma)\in M\times G$.

 Clearly we may cover $M\times M$ by open sets $U \times V$ such that there exists a smooth local section $\omega:U \cup V \to G$, i.e. $\omega_z \cdot o = z$ for $z \in U\cup V$. For such $U,V$, consider the diagram
$$
\begin{CD}
T^*M \times T^*M \times G_o \supset & \pi\inv(U) \times \pi\inv(V) \times G_o @>{\tilde \Phi}>> & (X,Y)\inv(U\times V) \subset\tilde \FFF\\
 & @V(X,Y,\bar \Gamma)VV & @VV(X,\Gamma)V \\
& U \times V \times G_o @>\Phi>> & U\times G
\end{CD}
$$
where, abbreviating $x:= \pi(\xi), y:= \pi(\eta)$ for $\xi, \eta\in T^*M$, the vertical map on the left is the projection,  the vertical map on the right is 
$$
(\xi,\eta,\gamma)\mapsto (x,\gamma),
$$
and
\begin{align*}
\tilde \Phi(\xi,\eta, \bar \gamma) &:= (\xi,\eta, \omega_x \bar \gamma \omega_y\inv) \\
\Phi(x,y,\bar \gamma) &:= (x,  \omega_x \bar \gamma \omega_y\inv) .
\end{align*}
Thus $(X,Y)\circ \tilde \Phi = \Phi \circ(X,Y,\bar \Gamma)$.
By definition of the fiber product of currents (cf. Section \ref{sect:manifold conventions}),
$$
(\D(f) \times \D(g) \times_{\F} G_o )\with (X,Y)\inv(U\times V)  = \tilde \Phi_*(\D(\restrict f U) \times \D(\restrict g V) \times G_o).
$$

Note that $\Phi$ is a diffeomorphism onto its image, with inverse
\begin{equation}
\Phi\inv(x,\gamma) = (x, \gamma\inv x, \omega_x\inv \gamma \omega_y).
\end{equation}
 and preserves orientation by the convention \eqref{eq:orientation convention}. Thus for a.e. $(x,\gamma) $ lying in this image, Theorem 4.3.2(6) and Theorem 4.3.5 of \cite{gmt}, together with the \MA~ axiom ($4'$) for $f,g$, imply that
\begin{align*}
 (-1)^{d\dim G}\langle\langle \D(f) \times \D(g) \times_{\F}& G_o , \Gamma, \gamma\rangle, X, x \rangle =   \\
&= (-1)^{d\dim G} \langle \D(f) \times \D(g) \times_{\F} G_o , (\Gamma,X), (\gamma, x)  \rangle \\
&= \langle \D(f) \times \D(g) \times_{\F} G_o , (X,\Gamma), (x,\gamma)  \rangle \\
 &=\tilde \Phi_*\langle \D(\restrict f U) \times \D(\restrict g V) \times G_o, (X,\Gamma)\circ \tilde \Phi, (x,\gamma) \rangle \\
 &= \tilde \Phi_*\langle \D(\restrict f U) \times \D(\restrict g V) \times G_o,  \Phi\circ (X,Y,\bar\Gamma), (x,\gamma) \rangle \\
&=   \tilde \Phi_*\langle \D(\restrict f U) \times \D(\restrict g V) \times G_o,   (X,Y,\bar\Gamma), (x, \gamma\inv x, \omega_x\inv \gamma \omega_y) \rangle \\
&=   \tilde \Phi_*(\delta_{(x,df(x))}\times\delta_{(\gamma\inv x, dg(\gamma\inv x))}\times \delta_{\omega_x\inv \gamma \omega_y})\\
&= \delta_{(x,df(x))}\times\delta_{(\gamma\inv x, dg(\gamma\inv x))}\times \delta_ \gamma,
\end{align*}
which is \eqref{eq:sum slice}.
\end{proof}

\subsection{Conormal cycles for 
\WDC~ sets} \label{sect:regular domains}   We give an abbreviated account of the \DC~ case of the theory presented in  \cite{fu94}. It is remarkable that the ad hoc hypotheses needed to make \cite{fu94} work out are always fulfilled here.

If $f \in \DC(M)$ is an aura for a set $A = f\inv(0) \subset M$, we then make the following provisional definition, which will be superseded in view of Theorem \ref{thm:n def} below:
\begin{equation}\label{eq:def nf}
N^*(f):= \proj_* \partial (\D(f) \with \pi\inv A)= -\proj_*\partial (\D(f) \, \with \,\pi\inv (M\setminus A)).
\end{equation}
This object is clearly defined locally, in the  sense that 
\begin{equation}
N^*(\restrict f U) = N^*(f) \with \pi\inv (U)
\end{equation}
for any open subset $U \subset M$.

\begin{thm}[\cite{fu94}] \label{thm:properties of n*} Under the conditions above, $N^*(f)$ is an integral current of dimension $d-1$ in $S^*M$, with
$$
\partial N^*(f)=N^*(f )\, \with \, \alpha = 0. \quad \quad \square
$$ 
%Given a \DC~ aura $(f,c)$ for a set $A\subset M$, one may construct a Legendrian cycle $N(f,c) \in \mathbb I_{d-1}(S^*M)$ by slicing the differential cycle of $f$ by a length function on $T^*M$ at a suitably small value (\cite{fu94}, Definitions 1.4) and projecting to $S^*M$.
%
\end{thm}

We show that this current depends only on the underlying set $A$:
\begin{thm}\label{thm:n def} If $A\subset M$ is a \WDC~ set, and $f,g$ are  \DC~auras for $A$, then 
$$
 N^*(f) =N^*(g).\quad \quad \square
$$
\end{thm}

\begin{definition}\label{def:n*A} {\rm We define this common value to be $N^*(A)$.}
\end{definition}

%
%The proof of Theorem \ref{thm:n def} relies on the following assertion, which is a direct consequence of the proof of Proposition \ref{prop:max dc} below. For another approach cf. \cite{PR}.
%
%\begin{lem} Let $A$ be a \WDC~ subset of $M$. Let $x \in A$ and $U\subset M$ a neighborhood of $x$. Then there exists a \WDC~ subset $B \subset U$ such that $B \subset A$ and $B\cap V = A \cap V$ for some open neighborhood $V$ of $x$.
%\end{lem}
%\begin{proof} Taking local coordinates, we may assume that $M =\rd$. Let $B'\subset U$ be a closed ball containing $x$. By the proof of Proposition \ref{prop:max dc} below, there exists a subset $E$ of full measure in some neighborhood of the identity in $\barsod$ such that $g B' \subset U$ for $g \in E$, and $gB' \cap A$ is a \WDC~ set.
%\end{proof}

\begin{proof}[Proof of Theorem \ref{thm:n def}] We show that  $N^*(f), N^*(g)$ agree when restricted to any coordinate neighborhood $(\psi,U)$. Let $V:=\psi(U)\subset \Rd$, and consider the \DC~ functions $f\circ \psi\inv, g\circ \psi\inv:V\to \R$, which are both auras for $\psi (A\cap U)\subset V$, with
\begin{align*}
 \tilde \psi\inv_*N^*(f\circ \psi\inv ) &= N^*(f)\with \pi\inv(U), \\
 \tilde \psi\inv_*N^*(g\circ \psi\inv) &= N^*(g)\with \pi\inv(U).
\end{align*}
Thus it is sufficient to show that every point $x \in V$ admits a neighborhood $W\subset V$ such that
$$
N^*(f\circ \psi\inv )\with \pi\inv(W) =  N^*(g\circ \psi\inv)\with \pi\inv(W).
$$

Put $r>0$  for the distance from $x$ to the complement of $V$ in $\Rd$, and let $h(y):= \max (0, |y|-\frac r 2)$ denote the standard aura for the closed disk $ B: =\bar B(0,\frac r 2)$.
By the proof of Proposition \ref{prop:max dc} below , for a.e. euclidean motion $\gamma$ the functions $(f\circ \psi\inv) +\gamma^*h, (g\circ\psi\inv)+ \gamma^*h$ are both auras for $\psi(A\cap U) \cap \gamma B$. Clearly such $\gamma $ may be chosen so that $x$ lies in the interior of $ \gamma B$, and hence $\gamma B \subset V$. Therefore these functions may be extended to auras for $\psi(A\cap U) \cap \gamma B$, considered as a  subset of $\Rd$.

It follows from Theorem 1.2 of \cite{PR} that 
$$
N^*((f\circ \psi\inv) + \gamma^*h) = N^*((g\circ \psi\inv) + \gamma^*h).
$$
Since the restrictions of these currents to the interior of $\gamma  B$ agree with those of $N^*(f \circ\psi\inv), N^*(g\circ\psi\inv)$ respectively, this completes the proof.
\end{proof}

\begin{prop}\label{prop:support of n*} If $A$ is a compact \WDC~ set with aura $f$, then
$$
\spt N^*(A) \subset \noreps f
$$
for all sufficiently small $\eps >0$.
\end{prop}
\begin{proof} This follows at once from Lemma \ref{lem:support df}.
\end{proof}

\subsubsection{Conic cycles} We will need the following alternative construction of the conormal cycle, a restatement of Proposition 1.3 and equations (1.3d), (1.3g), (1.4c) of \cite{fu94}. We identify $S^*M $ with $\ell\inv(1)$ and define the maps
\begin{align*}
\nu: T^*M \setminus \z \to S^*M, &\quad \nu(\xi):= \frac{\xi}{\ell(\xi)}, \\
m: \R \times S^*M \to T^*M, & \quad m(t,\bar \xi):= t \xi, \\
m_t: T^*M \to T^*M & \quad m_t(\xi):= t \xi, t \in \R, \\
z: = \z&:M \to T^*M
\end{align*}

\begin{prop}\label{prop:cone structure} Let $f \in \DC(M)$ be an aura for $A \csubset M$. Suppose $U$ is a neighborhood of $A$ and $r_0>0$ is small enough that $\ell(\xi) >r_0$ whenever $\pi(\xi ) \in U\setminus A$. Then 
\begin{align}
\label{eq:cone 1}\D(f) \with (\pi\inv(U) \cap \ell\inv[0,r_0]) &= z_*A + m_*([0,r_0] \times N^*(A)), \\
\label{eq:cone 2}N^*(A) &= \nu_*\langle \D(f)\with \pi\inv (U), \ell ,r\rangle \quad \text{ for all } r \in (0,r_0),\\
\label{eq:cone 3}\vec N^*(A):= z_*A + m_*([0,\infty) \times N^*(A)) &= \lim_{t\to \infty} m_{t*} (\D(f) \with \pi \inv (U)),\\
\label{eq:cone 4}N^*(A)&= \langle \vec N^*(A),\ell, 1\rangle.
 \quad \quad\square
\end{align}
\end{prop}

{\bf Remark.} Note that the compactness of $A$ ensures that such $U,r_0$ exist. The limit in \eqref{eq:cone 3}  exists in a particularly strong sense: for any $C>0$, there exists $t_0 $ such that for $t >t_0$ the restrictions to $\ell\inv[0,C)$ of the left hand side and the expression under the limit agree.

\section{The size of the Clarke differential of a \DC~function}
%
%Let $(f,0)$ be a DC aura on a Riemannian manifold $M$ of dimension $d$.
%
%We shall use the notation $\pi: T^*M\to M$ for the canonical projection and 
%$$\graph\partial f:=\bigcup_{x\in M}\partial f(x)=\{(x,u):\, x\in M,\, u\in\partial f(x)\}\subset T^*M$$
%for the graph of the Clarke differential.

The main result of this section follows. It is a sharpened version of Proposition 7.1 of \cite{PR}.

\begin{thm}\label{WDCbundle}
Let $f$ be a  DC aura on a Riemannian manifold $M$ of dimension $d$ and $\eps >0$. Then ${\noreps}f$ has locally finite $(d-1)$-content. 
\end{thm}
 
Using local coordinates it is clearly enough to prove the theorem in the Euclidean case, so we shall assume throughout this section that $M=\R^d$.
We follow the scheme of Pavlica and Zaj\'\i\v cek \cite{PZ}, relating a fundamental result about the boundary structure of convex bodies (due to Ewald, Larman and Rogers \cite{ELR}) to the set of tangents to the graph of a DC function. The classical result of \cite{ELR} is enough to establish the conclusion of Theorem \ref{thm:n def} in case the ambient manifold $M = \Rd$ (viz.\ Proposition~7.1 of \cite{PR}). However, for our present purposes we need the following more detailed version, where we use the notation of Section \ref{sect:convex}:

\begin{lem}\label{lem:sharp elr} Let $A,B$ be compact convex subsets of $\R^d$. Then the set
$$
\Sigma_{A,B}:=\left\{\left(x-y,\xi \right)\in \Rd \times S^{d-1}:  \xi \in \nor(A,x) \cap \nor(B,y) \right\}
$$
has finite $(d-1)$-content.

%If $A,B$ are general closed convex subsets of $\rd$ then $\Sigma_{A,B}$ has locally finite $(d-1)$-content.
\end{lem}

Following the approach of \cite{PZ}, we will combine this Lemma with a duality between the space of tangents to graphs of DC functions $f-g$ and the structure of $\Sigma_{A,B}$, with $A,B$ taken as the epigraphs of the  convex conjugates (Legendre transforms) of $f,g$. This will establish  a natural overestimate of the size of the support of the differential cycle of a DC function; exploiting the enhancements introduced in Lemma~\ref{lem:sharp elr}, this sharpens Proposition~7.1 of \cite{PZ}  (see Lemma~\ref{duality}). A simple argument then shows that this yields the desired estimate of the Minkowski content of $\noreps f$ for DC auras $f$.

{\bf Remark.} The corresponding lemma of \cite{ELR} was weaker in two senses: it applies only to the projection of $\Sigma_{A,B}$ to the first ($\Rd$) factor, and  it concludes only that this projection has  locally finite $(d-1)$-dimensional Hausdorff measure. As an aside from our main application, we present at the end of this section a proof of a sharper version of the main theorem of \cite{ELR} incorporating both of our improvements.

%
%Though we do not need it for our application, we observe here that 
%Theorem A, a sharpened version of the main theorem of \cite{ELR}, may be deduced from Proposition~\ref{lem:sharp elr} . We present a proof at the end of this section. 

%
%shall use results and ideas from three sources. The most important ingredient, Proposition~\ref{dim_segm} (and Corollary~\ref{noncompact}), is in fact a slight extension of the main result of Ewald, Larman and Rogers (\cite[Theorem~1]{ELR}) about the measure of directions of line segments lying on the boundary of a convex body. Here we additionally consider a normal direction of each segment, which is possible by applying two auxiliary lemmas (Lemmas~\ref{cover}, \ref{diameter}). 

%
%The latter was deduced from a critical lemma of \cite{ELR}, key to the proof of their theorem  that the Hausdorff dimension of the space of directions of line segments lying in the boundary of a general convex body in $\Rn$ is $n-2$. We prove Theorem \ref{thm:support} as an application of the following sharpened version of the Ewald-Larman-Rogers lemma.

\subsection{Normals to pairs of convex sets}
In this section we prove Lemma~\ref{lem:sharp elr}.

\begin{lem}\label{cover}
Let $d\in\en.$ There is a constant $C_d$ such that for every $r>0$ the following statement holds:
If $K\subset\er^d$ is a convex body and $r\leq\wth K$ then $K$ can be covered by $M$ balls of radius $r$ such that $Mr^d\leq C_d\cdot\vol_d(K)$. 
\end{lem}

%***Cant we rely on ELR, or Schneider's book, for the proof of this Lemma? -- JF***

\begin{proof}
The lemma is a simple application of \cite[Lemma~7]{ELR}.
%, the proof will be divided into two steps.
%First we prove it for the case where $K$ is equal to a ball, which is easy and can be done by a standard covering argument as follows.
Let $\tilde C_d$ be a constant such that every ball $B\subset\er^d$ of a radius $\rho$ can be covered by $M$ balls of radius $r<\rho$ such that 
\begin{equation}\label{estballbyballs}
Mr^d\leq \tilde C_d\cdot \omega_d\cdot\rho^d.
\end{equation}

Fix a convex body $K\subset\er^d$ and $r\leq\wth K$. 
Using \cite[Lemma~7]{ELR} we can cover $K$ by balls $B_1,\dots,B_{M_1}$ of diameter $\delta:=\sqrt{d}\cdot\wth K$ with
\begin{equation}\label{estKbyballs}
M_1\cdot(\wth K)^d\leq 2^d\sqrt{d}\cdot d!\cdot\vol_d(K). 
\end{equation}
Using \eqref{estballbyballs} we can cover every ball $B_i$ by balls $B^i_1,\dots,B^i_{M_2}$ of diameter $r$ such that
\begin{equation}\label{ballbyballs2}
M_2\cdot r^d\leq \tilde C_d\cdot \delta^d\cdot \omega_d.
\end{equation}
Multiplying \eqref{estKbyballs} and \eqref{ballbyballs2} we obtain
$$
M_1\cdot M_2\cdot r^d(\wth K)^d\leq 2^d\sqrt{d}\cdot d!\tilde C_d (\sqrt{d})^d (\wth K)^d \omega_d\cdot\vol_d(K).
$$
This means that $K$ can be covered by $M:=M_1\cdot M_2$ balls of the form $B^i_j$, $i=1,\dots,M_1$, $j=1,\dots,M_2$, all of diameter $r$, with
$$
Mr^d\leq C_d\cdot\vol_d(K),
$$
where $C_d:=2^dd!\tilde C_d (\sqrt{d})^{d+1} V_d$.
\end{proof}

\begin{lem}\label{diameter'}
Let $K\subset\er^d$ be a convex body. Then for $\tilde K:=K+B(0,1)$, $0<t<1$, and $\nu\in S^{d-1}$, the spherical diameter of the set 
$$
N_{\nu,t}(\tilde K):= \bigcup_{x\in C(\tilde K,\nu,t)} \nor(\tilde K,x)
%\{n\in\nor(\tilde K,x):x\in C(\tilde K,\nu,t)\}
$$
is smaller than $2\sqrt{3t}.$
\end{lem}

\begin{proof} Let $\nu \in \nor(\tilde K, x_0)$, and $n\in\nor(\tilde K,x), x\in C(\tilde K,\nu,t)$. Then
$$
x_0\cdot \nu - t \le x\cdot \nu \le x_0\cdot \nu.
$$
Furthermore $x-n \in K$, and therefore $x-n + \nu \in \tilde K$. In particular 
$$
 x\cdot \nu + (1  -n\cdot \nu) =(x  -n+\nu)\cdot \nu \le x_0\cdot \nu,
$$
so
\begin{equation}\label{eq:n dot nu}
\cos \theta =n\cdot \nu \ge 1-t
\end{equation}
where $\theta$ denotes the spherical distance between $n,\nu$. Since
$\cos \sqrt{3t} < 1-t$ for $0<t<1$, we find that $\theta < \sqrt{3t}$, and the conclusion follows.
\end{proof}

We are now ready to prove Lemma~\ref{lem:sharp elr}. The proof is an enhancement of a part of the proof of Theorem~1 from \cite{ELR}.

%\begin{prop}  \label{dim_segm}
%Let $M,L$ be two convex bodies in $\er^d$. Then the set
%$$
%N(M,L):=\{(x-y,p):p\in \nor(M,x)\cap\nor(L,y)\}
%$$
%has Minkowski dimension $d-1.$
%\end{prop}

%\begin{prop}  \label{dim_segm}
%Let $M,L$ be two convex sets in $\er^d$. Then the set
%$$
%N(M,L):=\{(x-y,p):p\in(\nor(M,x)\cap\nor(L,y)\}
%$$
%has $\sigma$-finite $(d-1)$-dimensional  Minkowski content.
%\end{prop}

\begin{proof}[Proof of Lemma~\ref{lem:sharp elr} ]
%We shall show that if $A,B$ are convex and compact then $\Sigma_{A,B}$ has finite $(d-1)$-dimensional Minkowski content. It is easy to see that this will imply the statement.
%It is clearly sufficient to prove the first statement, hence, assume that $A,B\subset\rd$ are compact and convex.

First note that the set $\Sigma_{A,B}\subset \Sigma_{\tilde A,\tilde B}$ with $\tilde A:=A+B(0,1)$ and $\tilde B:=B+B(0,1)$ 
(in fact, equality holds, but we will not need it in our proof).
Indeed, if we choose $(x-y,p)\in \Sigma_{A,B}$ with $p\in \nor(A,x)\cap\nor(B,y)$, then 
$$
p\in(\nor(\tilde A,x+p)\cap\nor(\tilde B,y+p)
$$ 
and, of course, $(x+p)-(y+p)=x-y$.
Define $K:=\tilde A+\tilde B.$

By \cite[Lemma~5]{ELR}, there exists  $\eps_0>0$ such that for every $0<\eps<\eps_0$ 
there are $n_1,\dots,n_M\in S^{d-1}$, $t_1,\ldots,t_M>0$ and a constant $D'=D'(d,A,B)$ such that
\begin{eqnarray}\label{inequalities}
\partial K\subset\bigcup_{i=1}^{M} C(K,n_i,t_i),\quad\quad\sum\limits_{i=1}^{M} \vol_d(C(K,n_i,t_i))\leq D'\eps,
\end{eqnarray}
and
\begin{equation}\label{width}
\wth C(K,n_i,t_i)\in[2\eps,36d\eps]
\end{equation}
for every $i$. We shall assume that $\eps_0<\tfrac{1}{72d}$, ensuring that these widths are all $\leq \frac 1 2$.
Since 
\begin{equation}\label{KplusL}
K=(A+B)+B(0,2)
\end{equation}
and by the fact that $\wth C(K,n_i,t_i)\leq\tfrac 12$ for every $i$, one can see that $\wth C(K,n_i,t_i)=t_i$ for every $i$. 
Hence \eqref{width} is equivalent to
\begin{equation}\label{width2}
t_1,\dots,t_M\in[2\eps,36d\eps].
\end{equation}

Select points $x_i\in C(A,n_i,t_i)$, $y_i\in C(B,n_i,t_i)$, and denote $z_i:=x_i-y_i$. 
We  show that
\begin{equation} \label{E_crucial}
\Sigma_{A,B}\subset\bigcup_{i=1}^M\left[ (z_i+\Delta C(K,n_i,2t_i))\times(N_{n_i,t_i}(\tilde{A})\cap N_{n_i,t_i}(\tilde{B}))\right].
\end{equation}
Let $(u,p)\in \Sigma_{A,B}\subset \Sigma_{\tilde{A},\tilde{B}}$, i.e., $u=x-y$ with $x\in \tilde{A}$, $y\in \tilde{B}$ and $p\in\nor(\tilde{A},x)\cap\nor(\tilde{B},y)$. 
Then there exists $i\leq M$ such that
\begin{equation}  \label{E_cap}
x+y\in C(K,n_i,t_i).
\end{equation}
We claim that
\begin{equation}  \label{E_caps}
x\in C(\tilde{A},n_i,t_i),\quad y\in C(\tilde{B},n_i,t_i).
\end{equation}
If not, then we may assume for definiteness that $x\cdot n_i<h_{\tilde{A}}(n_i)-t_i$, in which case additivity of support functions gives
$$(x+y)\cdot n_i<h_{\tilde{A}}(n_i)-t_i+h_{\tilde{B}}(n_i)=h_K(n_i)-t_i,$$
 contradicting \eqref{E_cap}.
Since \eqref{E_caps} implies that $p\in N_{n_i,t_i}(\tilde{A})\cap N_{n_i,t_i}(\tilde{B})$,
in order to prove \eqref{E_crucial} it thus remains only to show that $u-z_i \in \Delta C(K,n_i,2t_i)$.
To this end we note that  \eqref{E_caps} implies that
$$x-y\in C(\tilde{A},n_i,t_i)-C(\tilde{B},n_i,t_i)$$
and so
\begin{eqnarray*}
u-z_i=(x-y)-(x_i-y_i)&\in&\Delta(C(\tilde{A},n_i,t_i)-C(\tilde{B},n_i,t_i))\\
&=&\Delta(C(\tilde{A},n_i,t_i)+C(\tilde{B},n_i,t_i))\\
&=&\Delta C(K,n_i,2t_i),
\end{eqnarray*}
as claimed. Here we have used that
$$\Delta(U-V)=\Delta(U+V),\quad U,V\subset\rd,$$
and
$$C(K_1,n,t_1)+C(K_2,n,t_2)=C(K_1+K_2,n,t_1+t_2),$$
$K_1,K_2$ convex bodies, $n\in S^{d-1}$, $t_1,t_2>0$.

Given $1\leq i\leq M$, put $H_i:=\{x:\, x\cdot n_i=h_K(n_i)-2t_i\}$ for the hyperplane containing the base $P_i$ of the cap $C(K,n_i,2t_i)$. In fact this base may be described either as the intersection $H_i \cap K$, or alternatively as the projection of the cap onto $H_i$, i.e.
$$
\Pi_{H_{i}}(C(K,n_i,2t_i))=K \cap H_i =C(K,n_i,2t_i)\cap H_i =:P_i.
$$
If not, then there exists an interior point $x\in C(K,n_i,2t_i)$ such that the segment $[x,\Pi_{H_i}(x)]$ intersects the boundary of $K$ at some point $y\in\partial K$, and there  exists a unit outer normal vector $v\in\nor(K,y)$ with $v\cdot n_i\leq 0$. But since  by \eqref{KplusL} the unit ball is a Minkowski summand of $K$, the relation \eqref{eq:n dot nu} implies that $n_i\cdot v\geq 1-2t_i>0$, which is a contradiction.

By linearity, the same is true of the corresponding difference sets, i.e.
$$
\Pi_{\bar H_i}(\Delta C(K,n_i,2t_i))= \Delta P^i = \Delta C(K,n_i,2t_i)\cap \bar H_i,
$$
where $\bar H_i:= \{x:\, x\cdot n_i= 0\}$.
Clearly $\Delta C(K,n_i,2t_i)$ contains the union of two disjoint antipodal cones with common base $\Delta P^i$ and heights $2t_i$,  and therefore 
$$
\vol_d(\Delta C(K,n_i,2t_i)) \ge \frac {4t_i} d \vol_{d-1} \Delta P_i \ge \frac {8 \eps} d \vol_{d-1}( \Delta P_i ).
$$
Using  \eqref{inequalities} we conclude that 
\begin{equation}\label{eq:a}
\sum\limits_{i=1}^{M} \vol_{d-1}(\Delta P^i)\leq D:= \frac{D' d} 8.
\end{equation}

Using \eqref{KplusL} it is easy to see that $P_i$, and hence $\Delta P^i$ also, includes a ball of radius $\sqrt{t_i}$. Thus Lemma~\ref{cover} implies that $\Delta P^i$ may be covered by
$Q_i$ balls of diameter $\sqrt{t_i} \ge \sqrt{2\eps}$ such that 
\begin{equation}\label{eq:b}
Q_i 2^{\frac{d-1}2} \eps^{\frac{d-1}2}\le Q_i{t_i}^{\frac{d-1}2}\leq C_{d-1}\cdot\vol_{d-1}(\Delta P^i).
\end{equation}
Since
$
\Delta C(K,n_i,2t_i)\subset \Delta P^i+[-2t_i,2t_i]n_i
$
for every $i$, we conclude that every set $\Delta C(K,n_i,2t_i)$ can be covered by $Q_i$ balls of diameter $5\sqrt{t_i}$.
Denote these balls by $B_1^i,\dots,B^i_{Q_i}$.

By Lemma~\ref{diameter'} we have a fortiori
$$
\diam(N_{n_i,t_i}(\tilde A)\cap N_{n_i,t_i}(\tilde B))\leq 2\sqrt{3t_i}
$$
for every $i$. 
Therefore by \eqref{E_crucial} the set $\Sigma_{\tilde A,\tilde B}$ can be covered by $\sum_i^{M} Q_i$ sets of the form 
$$
(z_i+B_j^i)\times (N_{n_i,t_i}(\tilde A)\cap N_{n_i,t_i}(\tilde B))\subset \er^d\times S^{d-1},\quad i=1,\dots,M,\quad j=1,\dots,Q_i,
$$
each of which has diameter at most $7\sqrt{t_i}\leq 42\sqrt{d}\;\sqrt{\eps}$. Since by \eqref{eq:a} and \eqref{eq:b}
$$
({\sqrt\eps})^{{d-1}}\sum_i^{M} Q_i\leq  2^{-\frac{d-1} 2} C_{d-1} D,
$$
where the right hand side is independent of $\eps$, 
Lemma~\ref{lem:box}~(1) completes the proof.
\end{proof}

%As a side point, Lemma \ref{lem:sharp elr} implies the following sharpening of the main theorem of \cite{ELR} (I think). The original statement omits the coordinate $v$ and replaces Minkowski content with Hausdorff measure.
%\begin{cor} \label{cor:sharp elr} Let $A \subset \Rn$ be a closed convex set. Let $ U\subset S^{n-1}\times S^{n-1}$ denote the set of all pairs $(u,v)$ such that there exists a line segment $\sigma \subset \partial A$ in the direction $u$, with $v \in \nor(A,x)$ for all $x \in \sigma$. Then $U$ has $\sigma$-finite $(n-2)$-dimensional Minkowski content.
%\end{cor}

\subsection{The Minkowski dimension of $\noreps f$}
\begin{lem}  \label{duality}
If $f \in \DC(\Rd)$ then 
$$
\graph\partial f\subset T^*\R^d\cong \R^d\times\R^d
$$
has locally finite $d$-content.
\end{lem}

\begin{proof} This follows by an adaptation of the duality argument by Pavlica and Zaj\'i\v cek in \cite[Proposition 4.2]{PZ}. Let $f = g-h$ for convex functions $g,h$. 
We shall show that
$$
E_K(g,h):=\{(x,u-v):\, x\in K, u\in\partial g(x), v\in\partial h(x)\}
$$
has finite $d$-content for every $K\csubset\er^d$, which will imply the assertion by Lemma~\ref{lem:clarke of max}.

Suppose that both $g$ and $h$ are Lipschitz with a constant $L$ on an open neighbourhood $U$ of $K$ 
and let $g^*(x) := \sup_\xi (x\cdot \xi -g(\xi)) ,h^*(x)$ be the respective conjugate functions to $g,h$.
We may 
assume that both $g^*$ and $h^*$ are finite everywhere; this is equivalent to the assumption that
$
\bigcup_x\partial g(x) = \bigcup_x\partial h(x) = \Rd
$
(cf. \cite[Lemma 2.4]{PZ}), which in turn may be assured by altering $g,h$ outside of $U$ if necessary.
Thus  (\cite[Proposition~11.3]{RW04})
$$
\left(u\in\partial g(x)\;\;\&\;\; x\in K\right)\Longrightarrow \left(x\in\partial g^*(u)\;\;\&\;\; u\in B(0,L)\right)
$$
and similarly for $h$, so that
$$
E_K(g,h)\subset\{(x,u-v):\, u,v\in B(0,L),\, x\in\partial g^*(u)\cap\partial h^*(v)\}=:\tilde E_L(g,h).
$$

Let $A\subset \epi g^*, B \subset \epi h^*$ be compact convex subsets of $\R^{d+1}$ whose boundaries include the graphs of $\restrict{ g^*}{ B(0,L)}, \restrict{ h^*}{ B(0,L)}$ respectively.
Since
$$
\nor(\epi g^*, (u,g^*(u))) = \left\{\frac{(x,-1)}{\sqrt{1+|x|^2}}: x \in \partial g^*(u)\right\},
$$
%Using the relation 
%$$
%\left(x\in\partial g^*(u)\;\;\&\;\; u\in B(0,L)\right)\Longrightarrow \frac{(x,-1)}{\sqrt{1+|x|^2}}\in\nor(A,(u,g^*(u)))
%$$
and analogously for $h^*$, we find that the set $\Sigma_{A,B}$ from Lemma~\ref{lem:sharp elr} includes the set
$$\tilde\Sigma_{A,B}=\left\{\left(u-v,g^*(u)-h^*(v),\frac{(x,-1)}{\sqrt{1+|x|^2}}\right):\, u,v\in B(0,L),\, x\in\partial g^*(u)\cap\partial h^*(v)\right\},$$
which therefore
 has finite $d$-content.
Since $\tilde E_L(g,h)$ is a subset of the image of $\tilde\Sigma_{A,B}$ under the mapping
$$(a,t,b,s)\mapsto\left(-\frac bs,a\right),$$
which is Lipschitz on the set where $|b|^2 + s^2 =1, -\frac b s \in K$, this completes the proof.
\end{proof}

%If $f\in \DC(\Rd)$, with $f = g-h$ where $g,h$ are convex, we put
%$$\bar \partial f(x) : = \partial g (x) +(-\partial h(x) )$$
%where the addition is the  Minkowski sum. Note that this definition depends on the decomposition as a difference of convex functions. Clearly $\partial f(x)\subset \bar \partial f(x)$. One may then paraphrase Lemma \ref{duality} as: if $f \in \DC(\Rd)$ then the graph of $\bar \partial f$ has $\sigma$-finite $d$-dimensional Minkowski content.

\begin{proof}[Proof of Theorem~\ref{WDCbundle}]
Let $f$ be a DC aura in $\R^d$. 
We claim that the graph of $ \partial f$ includes a subset that is locally biLipschitz equivalent to the Cartesian product of $\noreps f$ with an interval. With Lemma \ref{duality} this implies the desired conclusion. 
%For simplicity we introduce a euclidean metric on $\Rd$, and hence on ${\Rd}^*$.

By the definition of ${\noreps}f$, if $\xi=(x,u)\in{\noreps}f$ then $tu\in\partial f(x)$ for some $t\ge \eps$. 
By Lemma \ref{lem:local extremum}, $0\in\partial f(x)$ whenever $f(x)=0$, so by the convexity of $\partial f(x)$ it follows that the whole segment $[0,\eps u]\subset \partial f(x)$. Thus the map
$$((x,u),t)\mapsto (x,tu)$$
 yields a biLipschitz embedding of $\noreps f \times [0,\eps]$ into $\graph\partial f$. Since the latter set has locally finite $d$-content the conclusion of Theorem \ref{WDCbundle} follows.
\end{proof}

\begin{proof}[Proof of Theorem~A] Clearly it is sufficient to prove the statement in the case where the convex set $K$ is compact.
Under this assumption, let $H,H'$ be distinct parallel hyperplanes that intersect $K$. Let $T_K^{H,H'}$ be the subset of $T_K$ from Theorem~A induced by boundary segments that intersect both $H$ and $H'$. Since clearly $T_K$ is the union of a countable family of subsets of such type, it will be sufficient to show that $T_K^{H,H'}$ has finite $(d-2)$-content for a fixed pair $H,H'$.

Denote $A=K\cap H$ and $B=(K\cap H')-z$, where $z\perp H$ is the vector with $H'=H+z$; $A,B$ are closed convex subsets of $H$. Observe that if $[x,\bar{y}]$ is a segment from the boundary of $K$ that intersects both $H$ and $H'$ and with direction $v$ and lying in a supporting hyperplane of outward normal direction $w$, then $(x-\bar{y}+z,\Pi_Hw/|\Pi_Hw|)$ belongs to one of the sets $\Sigma_{A,B},\Sigma_{B,A}\subset H\times S^{d-2}$, where $S^{d-2}\subset H$ is the unit sphere of $H$. Now  Lemma~\ref{lem:sharp elr} yields that $\Sigma_{A,B}$ has finite $(d-2)$-content. Inverting this mapping, we obtain that
$$(x-y,u)\mapsto\left(\frac{x-y+z}{|x-y+z|},\frac{|z|^2u-(u\cdot(x-y))z}{||z|^2u-(u\cdot(x-y))z|}\right)$$
maps $\Sigma_{A,B}\cup\Sigma_{B,A}$ onto $T_K^{H,H'}$. Since both denominators in the formula above are bounded from below (by $|z|,|z|^2$, respectively), the mapping is Lipschitz. It follows that $T_K^{H,H'}$ has finite $(d-2)$-content as well and the proof is finished.
\end{proof}

\section{Proof of Theorem B} \label{sect:kinematic}
Throughout this section we take $(M,G)$ to be a Riemannian isotropic space, i.e. $M$ to be a Riemannian manifold and $G$ a group of isometries of $M$ that acts transitively on the tangent sphere bundle $SM$. We choose base points $\bar o \in SM, o \in M$ with $\pi(\bar o) = o$, and denote by $G_{\bar o}, G_o\subset G$ the respective stabilizers of these points. Thus we may identify $SM \simeq G/G_{\bar o}, \ M \simeq G/G_o$. It is clear that the space $\Omega^*(SM)^G$ of $G$-invariant differential forms on $SM$ is isomorphic to the space $(\bigwedge^*T_{\bar o} SM)^{G_{\bar o}}$ of $G_{\bar o}$-invariant elements of the exterior algebra of the tangent space to $SM$ at $\bar o$. In particular, this space has finite dimension.

 For  $A \in \WDC(M)$ we let $N(A)$ denote the {\bf normal cycle} of $A$, i.e. the image of $N^*(A)$ under the diffeomorphism $S^*M \to SM$ induced by the Riemannian metric, and for a \DC~ aura $f$ and $\eps >0$ we define $\noreps f \subset SM$ similarly. We also continue to use the notation $\D(f)$ for the ``gradient cycle", the image in $TM$ of the differential cycle of $f\in \MA(M)$.

\subsection{Generic intersections}
The first part of Theorem B is established in the following.

\begin{prop}\label{prop:max dc}  Let $(M,G)$ be a Riemannian isotropic space and let $f,g$ be \DC~ auras for the compact sets $A,B \subset M$, respectively. Then there exists   $C \csubset G$ of measure zero  such that $h_\gamma := f + g \circ \gamma\inv$ is an aura for $A\cap \gamma B$ whenever $\gamma \notin C$. 

More precisely, every $\gamma_0 \in G\setminus C$ admits a neighborhood $W \subset G\setminus C$ with the following property: there exist open sets $ U \supset A, V \supset B$, and a constant $ \eps_0 >0$, such that 
$$
\ell (\xi + \gamma \eta) > \eps_0
$$
whenever (abbreviating $x:= \pi(\xi), y:= \pi(\eta)$)
\begin{align*}
\gamma &\in W,\\
\gamma y = x \in & (U\cap \gamma V) \setminus (A \cap \gamma B), \\
 \xi \in \partial f(x),& \quad \eta \in \partial g(y).
\end{align*}
\end{prop}
\begin{proof}

By definition of weak regularity, we may find $\eps >0$ and  neighborhoods $ U' \supset A,  V' \supset B$ such that
$\ell(\xi),\ell(\eta) \ge \eps$ whenever $\xi \in \graph (\partial f) \cap \pi\inv( U'\setminus A ), \eta \in \graph (\partial g) \cap \pi\inv( V' \setminus  B)$. Since $A,B$ are compact, it follows that $\noreps f, \noreps g$ are compact as well. 
 By Theorem \ref{WDCbundle} and Lemma \ref{lem:box} \eqref{item:mink prod}, the product $\noreps(f) \times s \noreps(g)$ is compact, with finite $(2d-2)$-content, where $d = \dim M$.

Put $s:SM \to SM$ for the fiberwise antipodal map.

 Consider
$$
F':= \{(\xi,\eta,\gamma) \in SM \times SM \times G: \gamma \eta = - \xi\}.
$$
The projection of $F'$ to the first two factors is clearly a fiber bundle with fibers diffeomorphic to
 $G_{\bar o}$. 
 Thus the  preimage of $\noreps f \times \noreps g$ under the projection of $F'$ is compact and has finite $(2d-2+ \dim G_{\bar o})$-content, so the projection $ C$ of this preimage to the third ($G$) factor has the same property. Since $\dim G = \dim SM + \dim G_{\bar o} = 2d-1 +\dim G_{\bar o}$, it follows that $C$ has measure zero in $G$.

Let $\gamma_0\in G\setminus C$. We prove the more detailed statement of the second paragraph. By construction, 
\begin{equation}\label{eq:antipodes disjoint}
\noreps f \cap s (\gamma_0 \noreps g)  = \emptyset.
\end{equation}
If the conclusion is false then  there exist $\xi_0, \eta_0 \in TM$ and sequences
$$
\gamma_i \to \gamma_0,  \quad \xi_i \to \xi_0, \quad \eta_i \to \eta_0
$$
such that, putting $x_i := \pi(\xi_i), y_i := \pi(\eta_i)$:
\begin{align*}
\xi_i &\in \partial f(x_i),\\
\eta_i &\in \partial g(y_i),\\
U' \setminus A \owns x_i& \to x_0 \in A\quad \text{(for definiteness)}, \\
V' \owns y_i &\to y_0 \in B, \\
\gamma_i y_i &= x_i,\\
\xi_i+ \gamma_i \eta_i &\to 0.
\end{align*}
Then $\ell(\xi_i) \ge \eps$, so $\xi_0 \in \partial f(x_0) \cap \ell\inv[\eps, \infty)$. Since  by continuity $\xi_0 + \gamma_0 \eta_0 = 0$, it follows that $\eta_0 \in \partial g(y_0) \cap \ell\inv[\eps, \infty)$, where $ \gamma_0 y_0 = x_0$. This contradicts \eqref {eq:antipodes disjoint}.
%
%  
%   then there exists $\eps_0 >0$ such that $\ell(\xi + \gamma_0 \eta) \ge 2\eps_0$  
% whenever $\xi \in \noreps f, \eta \in \noreps g$. 
% 
% Since $f,g$ are auras for $A,B$ respectively, we may choose open sets $V \supset A, V' \supset B$ and $\eps >0$ such that $\ell(\xi), \ell (\eta) \ge \eps$ whenever $ \pi(\xi) \in V \setminus A, \ \xi \in \partial f(\pi(\xi))$ and $ \pi(\eta) \in V'\setminus A, \ \eta\in \partial g(\pi(\eta))$.
%
% 
%It is clear that $A\cap \gamma B = h_\gamma\inv(0)$ for all $\gamma \in G$. If $\gamma \notin  C$ then we may take $U := V \cap \gamma V'$ in
% Proposition \ref{prop:transverse intersection} to conclude that $h_\gamma$ is an aura.
% The remaining assertions of the present Proposition follow easily by a simple adaptation of the proof of  Proposition \ref{prop:transverse intersection}.
 \end{proof}

{\bf Remark.} As a side point, a simpler version of this last proof yields the following, correcting an error in 
 \cite[Proposition~7.3]{PR} and \cite[\S2.2.3]{fu94}. There it is stated  that $f+g$ is an aura under  a condition  similar to but weaker than \eqref{eq:strong transverse}, with $\noreps f$ replaced by a larger set.
That statement may be true, but we do not know how to prove it.

\begin{cor}\label{cor:transverse intersection} Suppose $A,B \subset M$ are \WDC~ sets, with auras $f,g$ respectively. Suppose that for all sufficiently small $\eps >0$ 
\begin{equation}\label{eq:strong transverse}
\noreps f \cap s (\noreps g) = \emptyset.
\end{equation}
Then $A \cap B$ is a \WDC~ set, with aura $f+g$.
\end{cor}

\subsection{The main diagram}\label{sect:main diagram} Next we recall a classical construction of integral geometry, formalized current-theoretically in \cite{fu90}. Consider the space 
$$
\EE:=\{(\xi,\eta,\zeta,\gamma) \in SM^3 \times G: \pi(\xi) = \pi(\zeta) = \gamma \pi(\eta) \},
$$
to be thought of as the total space of a fiber bundle $\E$ over $SM \times SM$,with fiber over $(\xi,\eta) \in SM \times SM$ given by
$$
{\EE }_{\xi,\eta}:=\{(\zeta,\gamma)\in SM \times G: \pi(\zeta) = \pi(\xi) =\gamma \pi(\eta) \}\simeq S_oM \times G_o.
$$
The group of this bundle reduces to $G_{\bar o}\times G_{\bar o}$, acting on the model fiber $S_oM\times G_o$ by
\begin{equation}\label{eq:action}
(\gamma_0,\gamma_1)\cdot (\zeta, \bar \gamma) = (\gamma_0\zeta, \gamma_0\bar \gamma\gamma_1\inv).
\end{equation}
%which is the rationale for distinguishing this bundle notationally from the (family of) $G_o\times G_o$-bundle(s) $\F$.
There is then a double fibration
% is then the total space of the pullback $\mathcal E$ of this bundle to $SM \times SM$ via the map $\pi \times \pi$, and the map on the left is the projection of this bundle. Explicitly,
\begin{equation}\label{eq:main diagram}
\begin{CD} \EE @>(Z, \Gamma)>> SM \times G\\
@V{(\Xi, H)}VV \\
SM \times SM 
\end{CD}
\end{equation}
where $\Xi,H, Z,\Gamma$ are the restricted projections to the respective factors. The left action of $G\times G$ on $\EE$, given by
$$
(\gamma_0,\gamma_1)\cdot (\xi,\eta,\zeta,\gamma): = (\gamma_0 \xi, \gamma_1\eta, \gamma_0\zeta, \gamma_0\gamma \gamma_1\inv),
$$
intertwines the obvious  left actions of $ G\times G$ on $SM \times SM$ and $SM \times G$ respectively.
%Here the bottom map is $(\xi,\eta) \mapsto (\pi(\xi),\pi(\eta))$ and the map on the right is
%$$
%(\zeta, \gamma) \mapsto (\pi(\zeta), \gamma \pi(\zeta))=:(X(\zeta,\gamma),Y(\zeta,\gamma)).
%$$
%The space $E$ is
%$$
%E:= \{(\xi,\eta,\gamma): \gamma \pi(\eta) = \pi(\xi)\} \subset SM \times SM \times G
%$$
%and the maps  $\Xi, H,\Gamma$ are the projections to the three respective factors.

\subsubsection{The connecting current}
Each fiber $\EE_{\xi,\eta}$ includes the canonical subset
$$
{\bf C}_{\xi,\eta}:=\{(\zeta,\gamma) \in \EE_{\xi,\eta}: \zeta \in \overline{\xi,\gamma  \eta} \},
$$
where $\overline{\xi,\gamma  \eta}\subset S_{\pi(\xi)}M$ denotes the set of all points lying on a minimizing geodesic connecting ${\xi,\gamma  \eta}$. Typically this set is a geodesic arc, although if $\xi = \gamma \eta$ it is a point, and if $\xi = - \gamma \eta$ then it is all of $S_{\pi(\xi)}M$. The {\it interior} ${\bf C}_{\xi,\eta}^\circ$, consisting of those elements of ${\bf C}_{\xi,\eta}$ for which $\gamma  \eta \ne \pm\xi$ and $\zeta $ is an interior point of $\overline{\xi,\gamma  \eta}$, is then a smooth manifold diffeomorphic to $(0,1)\times (G_o\setminus\{\gamma: \gamma  \bar o \ne \pm \bar o\}) $.

The various $\CC_{\xi,\eta}$ may be identified with the model fiber
$$
{\bf C}:= \{( \zeta,\gamma) \in   S_oM \times G_o: \zeta \in \overline{\bar o, \gamma  \bar o}\} ,
$$
canonically up to the action \eqref{eq:action} of $G_{\bar o}\times G_{\bar o}$.
 Clearly ${\bf C}$ is compact and semialgebraic of dimension $1+\dim G_o$, hence  has finite volume of this dimension.
% The orientation of  $G_{ o}$    induces an orientation on ${\bf C}^\circ$, and thereby of each ${\bf C}^\circ_{\xi,\eta}$, as follows. It is clear that the union of the  ${\bf C}_{\xi, \eta}$ is a fiber bundle $\mathcal F$ over $SM \times SM$, with fibers isometric to 
%$$
%{\bf C}:= \{(\gamma, \zeta) \in G_o \times S_oM: \zeta \in \overline{\bar o, \gamma  \bar o}\} .
%$$
% 
We may take the diffeomorphism
 \begin{align}\label{eq:C current}
(0,1) \times  (G_o\setminus \{\gamma: \gamma  \bar o = \pm\bar o\})  &\to  {\bf C}^\circ \subset S_oM \times G_o, \\
\notag (t,\gamma) &\mapsto ( \proj ((1-t) \bar o + t \gamma  \bar o),\gamma)
 \end{align}
 to preserve orientations, thus endowing $\CC$ with the structure of an integral current with boundary
 \begin{equation}\label{eq:bdry C}
 \partial \CC =  p_{1*} G_o  -  p_{0*} G_o + K,
 \end{equation}
 where $ \supp K \subset  S_oM \times \{\gamma\in G_o: \gamma  \bar o = \pm\bar o\}) $ and $p_0,p_1:G_o \to S_oM\times G_o$ are given by
 \begin{align}
\notag  p_1(\gamma) &:=( \gamma  \bar o,\gamma ),\\
\label{eq:maps in bdry C} p_0(\gamma)&:= ( \bar o,\gamma).
 \end{align}

\begin{lem}\label{lem:main formal}
Fiber integration over $\CC$ yields a well-defined $G\times G$-equivariant operator
$$
\pi_{\CC*}: \Omega^*(\EE) \to \Omega^*(SM \times SM)
$$
of degree $-\dim \CC = - (1+ \dim G_o)$. In particular, if $\beta \in \Omega^*(SM)^G$ and $d\gamma$ is an invariant volume form for $G$ then there are constants $c_{i,j}$ such that
\begin{equation}\label{eq:}
\pi_{\CC *} (Z^*\beta \wedge \Gamma^* d\gamma) = 
\sum_{1\le i,j\le N} c_{i,j} \, \Xi^*\beta_i \wedge H^* \beta_j,
\end{equation}
 where $\beta_1,\dots, \beta_N$ constitute a basis for $\Omega^*(SM)^G$.
\end{lem}
\begin{proof} Cf. Section 1 of \cite{fu90}.
\end{proof}

The fiber integration operator gives rise to the following current theoretic construction. Let $S,T $ be currents living in $SM$. Then there is a well-defined  fiber product current $S \times T \times_{\E} \CC$ in $\EE$, given in any local trivialization by the corresponding Cartesian product. In particular, if $S,T$ are integral then so is this fiber product. The description of $\CC$ above gives rise to the following alternative local expression.

\begin{lem}\label{lem:mapping C} Let $S,T $ be currents living in $SM$, and let $W \subset G$ be an open set such that
$$
\gamma \in W \implies \spt S \cap s(\gamma\, \spt T) = \emptyset = \spt S \cap \gamma\, \spt T.
$$
Then
$$
(Z,\Gamma)_*(S \times T \times_{\E} \CC ) \with \Gamma\inv(W) = 
 c_*\left[ (S \times T \times (0,1) \times_{\F} G_o)\with \Gamma\inv(W)\right]
$$
where
$$
c(\xi,\eta, t,\gamma):  = (\nu((1-t)\xi + t\gamma\eta),\gamma). \quad \square
$$
\end{lem}

\subsection{The formal kinematic formula} \label{sect:slicing} We recall the main construction of \cite{fu90}, applied in the \WDC~ context. Let   $A,B\in \WDC( M)$. 
We put for a.e. $\gamma \in G$  
\begin{align}\label{eq:nor of intersect}
\mathcal J(A,B,\gamma)&:= (-1)^{d\dim G + d -1}Z_* \langle N(A)\times N(B) \times_{\mathcal E} \mathbf{C}, \Gamma, \gamma \rangle \\
\notag &+ N(A)\with \pi\inv (\gamma B) + N(\gamma  B) \with \pi\inv( A).
\end{align}

It will be useful to express the second and third terms of the right hand side of \eqref{eq:nor of intersect} in terms of slicing. 
%To this end we define
%\begin{align*}
%E_0&:=  SM \times M \times_{\F} G_o \\
%E_1&:= M \times SM \times_{\F} G_o.
%\end{align*}
As usual we denote by $\Gamma$ the projection maps of the spaces
$$
SM \times M \times_{\F} G_o , \quad  M \times SM \times_{\F} G_o
$$
 to $G$, and  put
$\Xi , H$ 
for the respective projections to the $SM$ factors.

\begin{lem}\label{lem:restrict as slice} For a.e. $\gamma \in G$
\begin{align}
\label{eq:slice1} \Xi_{*} \langle N(A) \times B \times_{\F} G_o, \Gamma,\gamma \rangle &= (-1)^{(d-1)\dim G}N(A) \with \pi\inv (\gamma   B), \\
 \label{eq:slice2} H_{*} \langle A \times N(B) \times_{\F} G_o, \Gamma,\gamma \rangle &= (-1)^{d +(d-1)\dim G}  N(  B) \with \pi\inv (\gamma\inv A)
\end{align}
\end{lem}

\begin{proof} We prove \eqref{eq:slice1}, the proof of \eqref {eq:slice2} being similar.
By \cite {gmt}, Theorem 4.3.8,  the left hand side of \eqref{eq:slice1} is equal to the restriction to $\pi\inv (\gamma B)$ of $\Xi_{*} \langle N(A) \times M \times_{\F} G_o, \Gamma,\gamma \rangle $. Thus we wish to show that this last current equals $(-1)^{(d-1)\dim G} N(A)$. 

Let $f$ be an aura for $A$. From \eqref{eq:Pi1}  we deduce that 
\begin{equation}\label{eq:slice df}
\Xi_*\langle \D(f) \times M \times_{\F}G_o, \Gamma, \gamma\rangle =
(-1)^{d\dim G} \D(f).
\end{equation}
Using \eqref {eq:slice boundary} and \eqref {eq:def nf} we calculate
\begin{align*}
\Xi_* \langle N(A) \times M \times_{\F} G_o, \Gamma,\gamma \rangle &= \Xi_{*} \langle \proj_*\partial (\D(f) \with \pi\inv(A)) \times M \times_{\F} G_o, \Gamma,\gamma \rangle \\
&= \proj_*\Xi_{*} \langle \partial (\D(f) \with \pi\inv(A)) \times M \times_{\F} G_o, \Gamma,\gamma \rangle \\
&= (-1)^{\dim G} \proj_*\Xi_{*}\partial \langle  (\D(f) \with \pi\inv(A)) \times M \times_{\F} G_o, \Gamma,\gamma \rangle \\
&= (-1)^{\dim G} \proj_*\partial\left[\Xi_{*} \langle  \D(f) \times M \times_{\F} G_o, \Gamma,\gamma \rangle \with \pi\inv(A)\right]\\
&=(-1)^{(d+1)\dim G} \proj_*\partial\left[\D(f) \with \pi\inv(A)\right] \quad \text{by \eqref{eq:slice df})}\\
&=(-1)^{(d+1)\dim G} N(A).
\end{align*}
\end{proof}

We now prove a formal version of the kinematic formula \eqref{eq:main kf}.

\begin{prop}\label{prop:formal kf} Let $\beta_0 \in \Omega^{d-1,G}(SM)$ be a $G$-invariant  form of degree $d-1$, and let $\beta_1,\dots,\beta_N\in \Omega^{d-1,G}(SM)$ be a basis for the space of all such forms. Then there are constants $c_{ij}, 1\le i,j\le N$, such that for any compact \WDC~ sets $A,B \subset M$ and bounded Borel measurable functions  $\phi,\psi:M \to \R$
\begin{align}
\notag\int_G \int_{\mathcal J(A,B,\gamma)} \pi^*(\phi \cdot (\psi \circ \gamma\inv))\wedge \beta_0 \, d\gamma &= \sum_{i,j} c_{ij} \int_{N(A)} \pi^*\phi \wedge \beta_i \cdot \int_{N(B)}\pi^*\psi\wedge \beta_j \\
\label{eq:formal kf} & +\int_A \phi \cdot \int_{N(B)} \pi^*\psi \wedge \beta_0 \\
\notag & + \int_B\psi \cdot \int_{N(A)} \pi^*\phi\wedge \beta_0
\end{align}
%
%
%\begin{align}
%\int_G \int_{\mathcal J(A,B,\gamma)\with \pi\inv(U \cap \gamma V) }\beta \ d\gamma
%&= \sum_{{1\le i,j\le N}}c_{ij}\int_{N(A)\with \pi\inv(U)} \beta_i \cdot \int_{N(B)\with \pi\inv(V)} \beta_j \\
%\notag  &  + \vol(V)   \cdot \int_{N(A)\with \pi\inv(U)} \beta + 
%\vol(U) \cdot \int_{N(B)\with \pi\inv(V)}\beta .
%\end{align}
\begin{proof} 
By the slicing theorem 4.3.2 (1) of \cite{gmt} and Lemma \ref{lem:restrict as slice}, the left hand side of \eqref{eq:formal kf} may be expressed as the sum 
\begin{align*}
(-1)^{d\dim G + d +1}\int_{N(A)\times N(B) \times_{\mathcal E} \mathbf{C}} & \Gamma^*d\gamma\wedge X^* \phi \wedge \, Y^* \psi \wedge  Z^* \beta_0 \quad + \\
&+ (-1)^{(d-1)\dim G} \int_{N(A) \times B \times_{\mathcal F} G_o} \Gamma^*d\gamma \wedge \Xi^*(\phi\wedge\beta_0) \wedge Y^*\psi  \\
 &+  (-1)^{d+ (d-1)\dim G}\int_{A \times N(B) \times_{\mathcal F} G_o} \Gamma^*d\gamma \wedge X^*\phi \wedge H^*(\psi \wedge \beta_0) \\
\end{align*}
corresponding respectively to the three terms in \eqref{eq:nor of intersect}, where (as usual) we have abused notation slightly in the labelling of the maps. By Lemma \ref{lem:main formal}, the first integral  may be expressed as 
$$
\sum_{1\le i,j\le N} c_{ij}\int_{N(A)} \pi^*\phi\wedge \beta_i \ \int_{N(B)} \pi^*\psi \wedge \beta_j.
$$

The second and third integrals become
$$
\int_{N(A) \times B \times_{\mathcal F} G_o}   \Xi^*(\phi\wedge\beta_0) \wedge Y^*\psi  \wedge \Gamma^*d\gamma
+(-1)^d \int_{A \times N(B) \times_{\mathcal F} G_o} X^*\phi \wedge H^*(\psi \wedge \beta_0)\wedge \Gamma^*d\gamma
$$
which by \eqref{eq:pi Go 1}, \eqref{eq:pi Go 2}
 yield the other terms on the right hand side of \eqref{eq:formal kf}.
\end{proof}

\end{prop}

\subsection{Conclusion of the proof of Theorem B}
Together with Proposition \ref{prop:formal kf}, the following concludes the proof of Theorem B.

\begin{thm}\label{thm:J=N reg} If $A,B \subset M$ are \WDC~ subsets of $M$ then
$$\mathcal J(A,B,\gamma) = N(A\cap \gamma B) $$
for a.e. $\gamma \in G$.
\end{thm}

By \eqref{eq:cone 4}, this follows from

\begin{lem}\label{lem:mu*J} For a.e. $\gamma \in G$
\begin{equation}\label{eq:vec n = z}
\vec N(A\cap \gamma B)  = z_*(A\cap \gamma B) + m_*([0,\infty) \times \J(A,B,\gamma)).
\end{equation}
\end{lem}

%\begin{lem}\label{lem:vec n int} For a.e. $\gamma \in G$
%\begin{equation}\label{eq:n slice}
%\vec N(A\cap \gamma B) =  (-1)^{d\dim G}\Sigma_*\langle \vec N(A) \times \vec N(B) \times_{\F} G_o, \Gamma, \gamma \rangle .
%\end{equation}
%\end{lem}
\begin{proof}
We claim first that
\begin{equation}\label{eq:n slice}
\vec N(A\cap \gamma B) =  (-1)^{d\dim G}\Sigma_*\langle \vec N(A) \times \vec N(B) \times_{\F} G_o, \Gamma, \gamma \rangle .
\end{equation}

To see this, let $f, g$ be auras for $A,B$ respectively,  let $C\csubset G$ be as in Proposition \ref{prop:max dc}, and let $\gamma_0 \in G\setminus C$. Let $W\owns \gamma_0, U \supset A, V\supset B$ be  neighborhoods as in Proposition \ref{prop:max dc}.   Then Theorem \ref{thm:ma addition} and Proposition \ref{prop:cone structure} imply that for $\gamma \in W$ 
\begin{align*}
 (-1)^{d\dim G}\vec N(A\cap \gamma B) &= (-1)^{d\dim G}\lim_{t\to \infty} m_{t*}(\D(h_\gamma)\with \pi\inv(U \cap \gamma V)) \\
&= \lim_{t\to \infty} m_{t*}  \Sigma_*\langle  (\D(f) \with \pi\inv(U)) \times  (\D(g) \with \pi\inv(V))\times_{\F} G_o,\Gamma,\gamma\rangle\\
&= \lim_{t\to \infty} \Sigma_*\langle m_{t*} (\D(f) \with \pi\inv(U)) \times m_{t*} (\D(g) \with \pi\inv(V))\times_{\F} G_o,\Gamma,\gamma\rangle\\
&= \Sigma_*\langle \vec N(A) \times \vec N(B) \times_{\F} G_o, \Gamma, \gamma \rangle. 
\end{align*}
Here the third equality is justified by the Remark following Proposition \ref{prop:cone structure}.

To complete the proof, we show that the right hand side of \eqref{eq:vec n = z} equals the right hand side of \eqref {eq:n slice}.

From the definition of $\vec N$, the current $\vec N(A) \times \vec N(B) \times_{\F} G_o$ may be expressed as the sum of the four terms
\begin{align}
\label{eq:00} z_*A &\times z_*B \times_{\mathcal F} G_o ,\\
\label{eq:01} z_*A &\times (m_*([0,\infty)\times N(B))) \times_{\mathcal F} G_o ,\\
\label{eq:10}  (m_*([0,\infty)\times N (A)))  &\times  z_*B\times_{\mathcal F} G_o ,\\
\label{eq:11}   (m_*([0,\infty)\times N (A)))  &\times    (m_*([0,\infty)\times N (B))) \times_{\mathcal F} G_o .
\end{align}
From \eqref{eq:Pi1} we deduce that for a.e. $\gamma \in G$ 
$$
\Sigma_*\langle z_*A \times z_*B \times_{\mathcal F} G_o ,\Gamma,\gamma \rangle  = (-1)^{d\dim G} z_*(A\cap \gamma B)
$$
and by the same reasoning the $\Sigma$ images of the $\Gamma$ slices of \eqref{eq:01}, \eqref{eq:10} yield, respectively,
\begin{align*}
(-1)^{d\dim G} m_*([0,\infty) &\times N(\gamma B)) \with \pi\inv ( A),\\
(-1)^{d\dim G} m_*([0,\infty) &\times N(A)) \with \pi\inv (\gamma B)
\end{align*}
 
It remains to show that the same operation, applied to \eqref{eq:11}, yields  
$$(-1)^{d\dim G}m_*([0,\infty) \times Z_*\langle N (A) \times N (B) \times_{\E} \CC,\Gamma,\gamma \rangle).
$$
It is enough to prove this for $\gamma \in W$, where $W \subset G$ is an open set as in Proposition \ref{prop:max dc}. In fact we prove the corresponding fact  in unsliced form, i.e. that the image of
\begin{equation} \label{eq:image 1}
(-1)^{d\dim G + d-1}[0,\infty) \times\left[ (N (A) \times N (B) \times_{\E} \CC )\with \Gamma\inv(W)\right]
\end{equation}
under the map $(t, \xi,\eta, \zeta,\gamma) \mapsto (t\zeta, \gamma)$, corresponding to the right hand side of \eqref{eq:vec n = z}, is identical to the image of
\begin{equation}  \label{eq:image 2}
(-1)^{d\dim G}\left[(m_*([0,\infty)\times N (A)))\times (m_*([0,\infty)\times N (B)))\times_{\F} G_o\right] \with \Gamma\inv(W)
\end{equation}
under the map $(\Sigma,\Gamma):(\xi,\eta,\gamma)\mapsto (\xi +\gamma \eta,\gamma)$, which corresponds to \eqref{eq:n slice}.

By the description \eqref{eq:C current} of the orientation of $\CC$, the image of \eqref{eq:image 1}  is equal to the image of 
$$
(-1)^{d\dim G + d-1}[0,\infty) \times N (A)  \times N (B)\times (0,1) \times_{\F} G_o
$$
under $(s,\xi,\eta,t,\gamma)\mapsto (s\nu((1-t)\xi + t\gamma \eta),\gamma )$,
or in other words equal to the image of
$$
(-1)^{d\dim G }[0,\infty) \times N (A)  \times (0,1)\times N (B) \times_{\F} G_o
$$
under $(s,\xi,t,\eta,\gamma)\mapsto (s\nu((1-t)\xi + t\gamma \eta),\gamma )$.

Meanwhile, the image of \eqref{eq:image 2} is equal to the image of
$$
(-1)^{d\dim G}[0,\infty) \times N (A) \times [0,\infty) \times N (B) \times_{\F} G_o
$$
under $(\sigma,\xi,\tau,\eta,\gamma)\mapsto (\sigma \xi + \tau \gamma \eta,\gamma )$. Since $\xi,\gamma \eta $ are linearly independent for $\gamma \in W$, it is easy to see that the map $(s,t) \mapsto (\sigma,\tau)$, where
\begin{align*}
\sigma &= \frac{s t}{|t\xi + (1-t)\gamma \eta|}\\
\tau&= \frac {s(1-t)}{  |t\xi + (1-t)\gamma \eta|}
\end{align*}
defines an orientation-preserving diffeomorphism between  $ (0,\infty) \times (0,1)\to (0,\infty)\times (0,\infty)$ (a modification of the standard polar coordinate map $(s,t) \mapsto (s\cos \frac{\pi t} 2,s\sin \frac{\pi t} 2)$). This completes the proof.
\end{proof}
\section{Questions and conjectures}
\subsection{Structure of WDC sets}  As mentioned in the Introduction, the class of \WDC~ sets includes the finite unions of semiconvex sets in general position as studied by \cite{zahle}, and also the boundaries of all convex bodies. However, the variety of geometric behavior  exhibited by \WDC~ sets seems clearly much broader than is displayed by these examples. It would be interesting to understand their behavior in more detail.

Here are some specific, and na\"ive, questions. 

\begin{enumerate}
\item Suppose $A \subset \Rd$ is a \WDC~ set such that the intrinsic volumes $\mu_{k+1}(A) = \dots =\mu_d(A) =0$. Is it true that $A$ is rectifiable of dimension $k$? Using Crofton's formula for \WDC~ sets (cf. \cite {PR}), it is not difficult to show that $\mu_k(A)$ equals the $k$-dimensional integral geometric measure of $A$. Thus Federer's structure theorem (\cite{gmt}, Theorem 3.3.13) implies that this question will be settled in the affirmative by showing that the $k$-dimensional Hausdorff measure of $A$ is no greater than $\mu_k(A)$. We expect that this should be resolvable by studying the relation between $A$ and a suitable carrier of $N(A)$.
%\item Not every \MA~ function is \DC, as shown by the well-known example of Hutchinson-Meier \cite{hm, fu alex}. 
%By the implicit function  theorem  for DC mappings (\cite{VZ}, Theorem 4.4), the epigraph of the Hutchinson-Meyer function $f_{\operatorname{HM}}$ cannot be a regular DC domain.
%%The \DC~ implicit function theorem is known \cite{?}: if $f:\Rd \to \R$ is a \DC~ function and $c$ is a Clarke regular value of $f$ then $f\inv(c)$ may be expressed locally as a graph of a \DC~ function of $d-1$ variables.
%%Thus the graph of the Hutchinson-Meier function $f_{HM}$ cannot be expressed as the level set of a \DC~ function at a Clarke regular value. 
%Is it possible, however, to express the epigraph of $f_{\operatorname{HM}}$ as a \WDC~ set?
%\item Suppose $A \subset \Rd$ is a \WDC~ set with nonempty interior. Is it necessarily true that there must exist a nonempty open subset of the boundary of $A$ that is expressible as the graph of a function? If so, must this function be \DC? (The answers to the corresponding questions are both yes in the positive reach/semiconvex setting.)
\item Is the distance function from a \WDC~ set $A$ necessarily \DC? Is it an aura for $A$? 
This would, in particular, provide us with a DC aura of $A$ which is semiconcave on the complement of $A$. 
Another natural question is then whether there is a DC aura for $A$ which is even smooth on $A^c$.
\item We call a set $A\subset\er^d$ {\bf locally WDC} if for every $x\in A$ there is $U_x$, an open neighborhood of $x$, and a WDC set $A_x$ such that $A\cap U_x=A_x\cap U_x.$
Is a locally WDC set necessarily WDC, for instance, are the $k$-dimensional DC~surfaces (see \cite[Section~3.1]{PR}) WDC? Note that they are locally WCD by \cite[Proposition~3.3]{PR}.

It is not difficult to see that the main points of the present paper (in particular the kinematic formulas) apply to locally WDC sets as well.
\item Is the generic projection of a WDC set again WDC?

\end{enumerate}

\subsection{Integral geometric regularity} The class of objects to which kinematic formulas apply seems impossible to describe in classical terms. On the other hand the quest to describe this class in {\it some} way is irresistible. One might attempt to come to terms with this situation as follows.

Let us say that a class $\mathcal C$ of compact subsets of $\Rd$ is an {\bf integral geometric regularity class} (or {\bf igregularity class}) if the following (probably redundant) list of axioms holds:
\begin{enumerate}
\item Every $A \in \mathcal C$ admits a normal cycle $N(A)$.
\item $\mathcal C$ is stable under diffeomorphisms $\phi$ of $\Rd$, with $N(\phi(A)) = \tilde \phi_*N(A)$, where $\tilde \phi:S^*\Rd \to S^* \Rd$ is the induced map.
\item If $A,B \in \mathcal C$ then for a.e. $\gamma \in \barsod$, the intersection $A \cap \gamma B\in \mathcal C$, with
$$
N(A\cap \gamma B) = \mathcal J(A,B,\gamma)
$$
where the right hand side is constructed formally from $N(A),N(B)$ as above.
\end{enumerate}
Thus the class of compact semiconvex sets is an igregularity class, and the present paper, together with \cite{PR}, implies that the same is true of the class of compact WDC subsets of $\Rd$. By \cite{fu94}, the class of compact subanalytic sets is an ``analytical igregularity class," i.e. the axioms above hold if (2) is replaced by stability under real analytic diffeomorphisms.

\begin{conjecture} There exists a unique maximal igregularity class of subsets of $\Rd$.
\end{conjecture}

This conjecture may be sharpened as follows. Consider the class $\mathcal N$ of compact sets $A$ with the property that
there exists a monotone sequence $M_1\supset M_2 \supset $ of compact smooth domains $\bigcap_n M_n=  A$, where the masses of $N_{M_n}$ are bounded by a fixed constant.

\begin{conjecture} $\mathcal N$ is an igregularity class.
\end{conjecture}

It is not even known whether every element of $\mathcal N$ admits a normal cycle. Simple examples show that it is not possible to obtain $N(A)$ as a subsequential limit of the $N(M_n)$. However, on naive geometric grounds it is natural to suppose that $N(A)$ could be constructed by some kind of pruning procedure from such a subsequential limit.

\begin{conjecture} $\mathcal N$ is the unique maximal  igregularity class of subsets of $\Rd$.
\end{conjecture}

However, it is not completely clear whether  WDC sets belong to $\mathcal N$.

It also seems possible that the approach of \cite{fu94} might  be made to work in greater generality, as suggested by the following sample statement:

\begin{conjecture} Let $f \in \MA(\Rd)$. Then $f\inv(-\infty,c] \in \mathcal N$ for a.e. $c \in \R$.
\end{conjecture}

\subsection{The Weyl tube principle} Until this point we have not mentioned one of the most striking phenomena of integral geometry, often referred to as ``the Weyl tube formula": if $M \subset \Rd$ is a smoothly embedded submanifold then the Federer curvature measures of $M$ are Riemannian invariants. In particular, these measures may be constructed from the structure of $M$ as an inner metric space. It is easy to show that the latter formulation holds also if $M$ is a (not necessarily smooth) compact convex set. 

Does it hold also for general WDC sets? This is not known even in the semiconvex case, nor for the case of boundaries of convex sets. The sole exception in the latter framework is the case where the ambient dimension $d= 3$, in which case the boundary of a closed convex set $A \subset \R^3$ is known to be a ``manifold of bounded curvature" in the sense of Alexandrov (cf. \cite{resh}).

The \WDC~ case presents a still more primitive obstacle: we do not know whether a general connected \WDC~ set $A$ is always a length space, i.e. whether any two points $x,y \in A$ are always joined by a rectifiable path $\subset A$. This would follow if we knew that such $A$ is a Lipschitz neighborhood retract in the sense of \cite{gmt}; although the proof of  Proposition 1.2 of \cite{fu94} assumes that this is so, this assertion is unjustified at present.

\end{document}